\documentclass[12pt]{article}

\usepackage{amsmath,amssymb,latexsym}
\usepackage{bbm} 
\usepackage{epsfig}
\usepackage{lscape}
\usepackage{array}

\usepackage{bbm} 
\usepackage[utf8]{inputenc}  
\usepackage[T1]{fontenc}
\usepackage{moreverb}
\usepackage{cleveref}

\def\bI{\boldsymbol{I}}

\def\bO{\boldsymbol{O}}

\def\bzero{\boldsymbol{0}}

\def\real{\mathbb{R}}

\def\T{\mathbb{T}}

\def\relative{\mathbb{Z}}

\def\B{\mathbb{B}}

\def\U{\mathbb{U}}

\def\d{\mathbbm{d}}
\def\x{\mathbbm{x}}
\def\y{\mathbbm{y}}

\def\T{\mathbb{T}}
\def\I{\mathbb{I}}
\def\mat1{\mathbb{I}}

\def\0I{\overset{0}{I}}

\def\limk{\lim_{k \rightarrow \infty}}

\newcommand { \Iletter}[1] {I\kern-0.10em #1 }

\def\bit{\begin{itemize}}
\def\eit{\end{itemize}}
\def\ben{\begin{enumerate}}
\def\een{\end{enumerate}}
\def\bde{\begin{description}}
\def\ede{\end{description}}

\def\bar{\begin{array}}
\def\ear{\end{array}}
\def\beq{\begin{equation}}
\def\eeq{\end{equation}}
\def\bfi{\begin{figure}[hbt] \begin{center}}
\def\efi{\end{center} \end{figure}}
\def\noi{\noindent}
\def\bce{\begin{center}}
\def\ece{\end{center}}
\newcommand{\proof}{{\bf Proof. }}
\newcommand{\cqfd}{\hfill $\Box$}

\newtheorem {thm} {Theorem} [section]
\newtheorem {maindifficulty} {Main difficulty} [section]
\newtheorem {prpstn} {Proposition} [section]
\newtheorem {rslt} {Result} [section]
\newtheorem {prprt} {Property} [section]
\newtheorem {ssmptn} {Assumption} [section]
\newtheorem {cnstrnt} {Constraint} [section]

\newtheorem{coro} {Corollary} [section]
\newtheorem{lem} {Lemma} [section]
\newtheorem{dfntn} {Definition} [section]
\newtheorem {numerotation} {Numerotation convention} [section]

\begin{document}

\title{Linear Programming Problem Solved By a Special Substitution Method}

\author{L. Truffet \\
  IMT Atlantque \\
  Departement Automation, Production and Computer sciences \\
  La Chantrerie, 4 rue A. Kastler 44300 Nantes, France \\
mail: laurent.truffet@imt-atlantique.fr} 

\maketitle

\begin{abstract}
  In this paper we develop a very special substitution method for solving
  a general linear programming problem (LPP). Of course the substitution is a kind of elimination
  of variable but this method must not be confused
  with the so-called Fourier-Motzkin elimination. The susbtitution developed in this
  paper only differs by the set of criteria that a variable must verify to be
  substitued. Most of the criteria are associated with the cost function of the
  LPP. We prove that the research of the criteria is strongly polynomial.
  Thus, the special substitution inehrits of the strong polynomiality which characterizes
  the classical substitution for linear systems. Moreover, as for the classical substitution the
  backward substitution for finding a vertex associated with the optimum is still valid and
  does not require to inverse a matrix.
  
\end{abstract}

\noi
{\bf Keywords}: Substitution, rewritting system, inequalities,
algorithmic complexity, strong polynomiality. \\
MSC: 90C05, 03D15 \\

\noi
NOTA:
  
  This new version clearly stresses the fact that we detail the sketch of the substitution method
  presented in ROADEF'26 (see \cite{truffet:hal-05534872}). The paper
  particularly refines the \textsc{nearest to zero interval criterion}
  briefly described in ROADE'26 for the cost function. Moreover, this new version of the
  paper point out the following two main points. (1). The concept of {\em dominating variable}
  is a copy-paste of the concept developed for $(\max,+)$-linear programming problem
  (see \cite{truffet:hal-05556396} and \cite{truffet2026substitutionminimizingmaximizingtropicallinear}).
  This concept only deals with positivity and strictly positivity property of linear functions. These two
  properties are trivialy verified for
  $(\max,+)$-linear functions and $(\max,+)$-affine functions, respectively. (2). Even if the concepts of $h$-boundedness
  are very similar in this paper and in $(\max,+)$-linear programming problem the global hierarchy
  between $h$-bounded functions and $h$-unbounded linear functions does not
  exist in the linear algebra. We have replaced it by conditioned hierarchy. 

This new version correct a bad logical mistake which have been done in the previous versions of
that preprint for the linear programming problem (LPP). Indeed, LPP problem deals with inequalities
describing linear half-spaces involved in the problem which are of the form $a+b -c-d \leq e$. Such inequalities
generate the disjunction of $\leq$-inequalities $(a \leq -b+c+d+e) \mbox{ or }(b \leq -a+c+d+e)$. Such inequalities
can also generate the disjunction of $\geq$-inequalities $(c \geq a+b-d -e) \mbox{ or } (d \geq a+b-c -e)$. Thus,
from the logical point of view the LPP deals with conjunction of a set of
disjunctions (ie. intersection of half-spaces) and not conjunction of a set of
conjunctions which was treated in the previous versions of the preprint. 

However, let us
    stress that this logical error
    does not question the general principle of the substitution method based on
    \textsc{nearest to zero} principle.

\section{General notations and definitions}
\label{introMainNotations}
For every $k \leq k'$, $k,k' \in \relative$ we define the 
discrete interval $[|k,k'|]:=\{k, k+1, \ldots, k'-1,k'\}$. 
The inclusion of sets is denoted $\subseteq$ and the strict inclusion 
of set is denoted $\subset$. If $L$ is a set then its complementary set
is denoted $\overline{L}:=\{l: l \notin L\}$. 
For any finite set $L$ the cardinal of $L$ is
denoted $\textsf{n}(L)$.

In the sequel of the paper we will use the following conventions/definitions dealing with vectors and matrices.

  \noi
  For any matrix $A$, $a_{i,.}$ denotes the $i$th row of $A$, $a_{.,j}$ denotes the
  $j$th column of $A$ and $a_{i,j}$ denotes the entry $(i,j)$ of $A$. \\

  The scalar operations are generalized to matrices as follows.

  \bit
\item Addition of two matrices: $(a_{i,j}) + (b_{i,j}) := (a_{i,j} + b_{i,j})$.
\item Product of two matrices: the entry $(i,j)$ of the matrix $C=A B$ is
  defined by: \\
  $c_{i,j}:= \sum_{k} a_{i,k}  b_{k,j}$.

  \item For every $m \times n$-matrix $A$ we define the submatrix of $A$ denoted $A_{IJ}$ by:
    \begin{equation}
      \label{eqSubMatNotation}
A_{IJ}:= (a_{i,j})_{i \in I, j \in J}.
    \end{equation}
    Where $I \subseteq [|1,m|]$ and $J \subseteq [|1,n|]$.

\item The oder relation on $\real$ or $\real^{k}, k \geq 2$ is denoted $\leq$ or $\geq$. In the case of
$\real^{k}, k \geq 2$ the order is the componentwise order between vectors.
  
\item Comparaison of two matrices: $A \leq B$ means $\forall i,j: a_{i,j} \leq b_{i,j}$.

  \item All vectors are column vectors. And the transpose operator is denoted $(\cdot)^{\intercal}$.

  \item For any vector $s$ we define its support as: 
    $\textsf{supp}(s):=\{i: s_{i} \neq 0 \}$.

\item If $s$ is a $n$-dimensional vector and
$C \subseteq [|1,n|]$, then $s_{C}$ denotes the $\textsf{n}(C)$-dimensional vector defined by: $s_{C}:=(s_{i})_{i \in C}$.

  \eit

  We will use the bold symbols for the following particular vectors and matrices.
  
  \bit
\item The boldsymbol $\bzero$ denotes the null vector which all components are $0$. The number of components
  of the vector is determined by the context.

\item The boldsymbol $\boldsymbol{+\infty}$ denotes the vector which components are all equal to $+\infty$.
  The number of components
  of the vector is determined by the context. 
  
\item The matrix $\bI_{n}$ denotes the $n \times n$-identity matrix. Its diagonal
  entries are all $1$ and its off-diagonal entries are all $0$.
\item $\bO_{m,n}$ will denote the $m \times n$-null matrix (ie, all entries are $0$). 
  When $m=n$ the the square null matrix $\bO_{n,n}$ is denoted $\bO_{n}$.
  
  \eit

  \section{Introduction}
  \label{secIntro}

\subsection{Problem statement and main result}
\label{subsubproblemgeneral}

We are interested in the linear programming problem (LPP) which has the following structure:

\begin{equation}
  \label{eqLPP}
\max\{z=c^{\intercal}x,\; x: Ax \leq b, x \geq \bzero\}
  \end{equation}
$x$ denotes the vector of variables, $A$ is a given matrix, $c$ a given cost vector and $b$ is a given column vector.
The dimensions of the LPP are specified in \S~\ref{introAllthematos}.

We have the following main theorem of the paper.

\begin{thm}
  \label{maintheo}
  There exists a special kind of substitution method which solves any general
  linear programming problem in strongly polynomial time. 
  \end{thm}

The reader must be aware that the substitution method developed in this paper must not be confused
with the so-called elimination method of Fourier-Motzkin. And the corollary of Theorem~\ref{maintheo} is

\begin{coro}
  \label{corSmale9th}
  If Theorem~\ref{maintheo} is proved then the Smale's $9$th problem \cite{kn:Smale98}
  dealing with linear programming is resolved.
  \end{coro}

\subsection{Related works}
One of the most famous methods for solving linear programs can be traced back to
Fourier (see eg. \cite{kn:Fourier1888}) and his elimination method for the variables of
the problem. This method has an exponential time of execution. Another very famous
method is the simplex method developped by Dantzig in the '40s. This method has also
been proved to be exponential (see eg. \cite{kn:KleeMint72}). Ellipsoid method and
interior point methods have been proved to be weakly polynomial (see eg. Khachiyan \cite{kn:Khachiyan79}
and Karmarkar \cite{kn:Karmarkar84}). Relaxation type method has also been proposed and proved to be
strongly polynomial for particular cases by Chubanov (see eg. \cite{kn:Chuba2011}, \cite{kn:Basuetal2013}
and references therein). Strongly polynomial algorithms have been developed for
special class of linear programs (see eg. \cite{kn:Megiddo83}, \cite{kn:AdlCosa91},
\cite{kn:Dadushetal2024} and refrences therein). The relevance of finding strongly polynomial
algorithm for linear programming in general is explained by eg. Schewe \cite{kn:Schewe09} because
Schewe proved that parity and mean payoff games reduce to linear programming problem.

\subsection{Vectors, matrices, sets and other materials involved in the LPP}
\label{introAllthematos}

We consider two integers $m \geq 1$ and $n \geq 1$. The real valued vectors involved in the LPP are:
the $m$-dimensional vector $b$, the $n$-dimensional vector $c$. The vectors of the
remaining variables are the $n$-dimensional vector $x=(x_{j})_{j=1}^{n}$ and the $m$-dimensional vector
$y=(y_{j})_{j=1}^{m}$.

The matrix involved in the LPP is: the real valued $m \times n$-matrix $A$, ie $A \in \real^{[|1,m|],[|1,n|]}$.

Noticing that the sets $\real_{+}^{n}$ and $\real_{+}^{m}$ are not stable
by linear combinations in general the sets involved (polyhedra and cones) are defined over
$\overline{\real}^{n}$ and $\overline{\real}^{m}$ and are the following ones. Where
$\overline{\real}:= \real \cup \{-\infty,+\infty\}$.

For the primal case the convex polyhedron $\mathcal{P}(A,b)$ defined as a set of inequalities:

\begin{equation}
  \label{eqdefPAb}
\mathcal{P}(A,b):=\{x \in \overline{\real}^{n}:\;  A x \leq b\}
\end{equation}
is invoved. And we associate with this polyhedron the cost function: $x \mapsto c^{\intercal} x$.

We also impose the following constraint concerning the maximum we are looking for.

\begin{cnstrnt}[Positive maximum]
  \label{positiveMaxonly}
The optimum we are looking for is always a maximum.
Moreover, the sign of this researched maximum is
always $\geq 0$.
\end{cnstrnt} 

This constraint is based
on the equality: ``$\max(c^{\intercal}x) =-\min(-c^{\intercal}x)$'' and the strong duality argument. If the initial LPP
does not find a positive maximum which could indicate a negative maximum. In this case we are looking at
$\min(-c^{\intercal}x)$ instead of looking at $-\min(-c^{\intercal}x)$. This $\min$-LPP can be transformed into
another $\max$-LPP by strong duality argument.
And we call improperly this other LPP the {\em dual LPP}. Thus the Constraint~\ref{positiveMaxonly} is
not restrictive at all and has the main advantage to allow us to use the Fourier's trick (see
eg \cite{kn:Williams86}, \cite{kn:Juhel2025}). 

The {\em dual LPP} deals with the convex polyhedron $\mathcal{P}(-A^{\intercal}, -c)$ 
\begin{equation}
  \label{eqdefPA'c'}
  \mathcal{P}(-A^{\intercal}, -c):=\{y \in \overline{\real}^{m}: -A^{\intercal} y \leq -c\}.
  \end{equation}
And the cost function which is associated with this polyhedron is: $y \mapsto -b^{\intercal}y$.

But our working space will be in fact the convex cone associated with 
the polyhedron $\mathcal{P}(A,b)$ by homogenization, ie:

\begin{equation}
  \label{eqdefCAb}
\mathcal{C}(A,b):=\{(x,h) \in \overline{\real}^{n} \times \overline{\real}:\;  A x \leq b h\}.
\end{equation}
assuming that the cone is not trivial, ie.:
\begin{equation}
  \label{hypCAbnontrivial}
\textbf{H}: \mathcal{C}(A,b) \neq \{\bzero \}.
\end{equation}

In this case the associated cost function is:
\begin{equation}
  \label{initialcostxh}
  \textsf{cost}: (x,h) \mapsto c^{\intercal}x + c_{h}h.
\end{equation}
  The cost
$c_{h}$ associated with the variable $h$ is set
to $0$ initially.

For the improperly called {\em dual LPP} our working space is the convex cone associated with 
the polyhedron $\mathcal{P}(-A^{\intercal}, -c)$ by homogenization, ie:

 \begin{equation}
    \label{dualCone}
    \mathcal{C}(-A^{\intercal}, -c):=\{(y,h') \in \overline{\real}^{m} \times \overline{\real}: -A^{\intercal} y \leq -c h'\},
    \end{equation}
assuming that:
\begin{equation}
  \label{dualhypCAbnontrivial}
\textbf{H'}: \mathcal{C}(-A^{\intercal}, -c) \neq \{\bzero \}.
\end{equation}

In this case the associated cost function is:
\begin{equation}
  \label{initialcost'yh'}
  \textsf{cost}': (y,h') \mapsto -b^{\intercal} y + b_{h'} h'.
  \end{equation}
The cost $b_{h'}$ associated with the variable $h'$ is set
to $0$ initially.

There exists a main advantage to work with cones. Indeed, by definition of a
cone $C$ (ie. $\forall a \in \real_{+}, \; x \in C \Rightarrow a x \in C)$
the hypothesis $(\textbf{H})$ (see (\ref{hypCAbnontrivial})) and $(\textbf{H'})$ (see (\ref{dualhypCAbnontrivial}))
induce the following fundamental assumption to be made:

       \begin{ssmptn}[$\lambda$-parametrized research of a finite maximum]
         \label{AsRlambda}
         All variables $(x,h) \in \real^{n+1}$ or $(y,h') \in \real^{m+1}$
         involved in the cones defined by (\ref{eqdefCAb}) or (\ref{dualCone}) can suppose to be at
         first sight elements of the following
         $\lambda$-parametrized symmetric interval:
          \begin{equation}
          \label{defRlambda}
[-\lambda, \lambda]
        \end{equation}
        for any $\lambda >0$.
         \end{ssmptn}

       And we also know when this assumption cannot be made.

       \begin{prpstn}[Limit of the Assumption~\ref{AsRlambda}]
         \label{propLambdaParamFausse}
         Under $(\textbf{H})$ (resp. $(\textbf{H'})$) the Assumption~\ref{AsRlambda} is false iff there exists
         $x_{j}$ (resp. $y_{j}$) such that $a_{.,j} \leq \bzero$ (resp. $-a^{\intercal}_{.,j} \leq \bzero$). 
         \end{prpstn}
       \proof
         If $\exists x_{j}: a_{.,j} \leq \bzero$. Then, under $(\textbf{H})$ to every finite
         element $x \in \mathcal{C}(A,b)$, $x \neq \bzero$, one can associate the non finite vector
         $x^{\infty}$, such that $x^{\infty}_{k}=x_{k}$ if $k \neq j$ and $+\infty$ otherwise. And this vector
         is also element of the cone $\mathcal{C}(A,b)$. Conversely, under  $(\textbf{H})$, any inequality
         $a_{i,.}x \leq b_{i}h$ admits an $x$ such that $x_{j}=+\infty$ as solution as soon as $a_{i,j} \leq 0$. And the
         result is proved.          
         \cqfd

         In the sequel it will be sometimes convenient to relax Assumption~\ref{AsRlambda} as follows.

         \begin{ssmptn}
         \label{AsRlambdaxj}
         For any index $j$ we can assume that the $j$th variable of the
         problem is element of $[-\lambda_{j}, \lambda_{j}]$ for some arbitrary $\lambda_{j} >0$ and
         all the other variables (homogenization variable included) can suppose to be
         elements of $[-\lambda, \lambda]$ for any $\lambda >0$.
         \end{ssmptn}

         \subsection{The general procedure for solving the LPP}
         \label{subGeneralProcedure}
Let us describe the general procedure for solving the LPP defined by
(\ref{eqLPP}). We will apply the special substitution method on the
problem called {\em Positive Maximum Research Problem}. We denote by:

\bit
\item \textsf{pmrp}($A,b,c,z,x,h$; $(m, n), \textsf{cost}$) the positive maximum research
  problem associated with the cone obtained by homogenization of the polyhedron
  $\mathcal{P}(A,b)$ and the cost function $z=\textsf{cost}(x,h):=c^{\intercal}x + c_{h}h$. Where
  $c_{h}$ is set to $0$ initially.

  \item \textsf{pmrp}($-A^{\intercal},-c,-b,z',y,h'$; $(n,m), \textsf{cost}'$) the positive maximum research
  problem associated with the cone obtained by homogenization of the polyhedron
  $\mathcal{P}(-A^{\intercal},-c)$ and the cost function $z'=\textsf{cost}'(y,h'):=-b^{\intercal} y + b_{h'}h'$.
  Where $b_{h'}$ is set to $0$ initially.
  
  \eit

  The LPP procedure is defined as follows.

  \begin{equation}
    \label{algogenproc}
    \bar{llll}
    \mbox{LPP} & := & \mbox{Case $1$}: & \mbox{\textsf{pmrp}($A,b,c,z,x,h$; $(m, n), \textsf{cost}$)
      returns $h \neq 0$} \\
    \mbox{} & \mbox{} & \mbox{} & \mbox{then $\exists \; \max \geq 0$} \\
    \mbox{} & \mbox{} & \mbox{Case $2$}: & \mbox{$h=0$ and \textsf{pmrp}($-A^{\intercal},-c,-b,z',y,h'$; $(n,m), \textsf{cost}'$)} \\
      \mbox{} & \mbox{} & \mbox{} &  \mbox{returns $h'\neq 0$} \\
\mbox{} & \mbox{} & \mbox{} &   \mbox{then $\exists \; \max \leq 0$} \\
   \mbox{} & \mbox{} &  \mbox{Case $3$}: & \mbox{$h=h'=0$ then LPP has no maximum}.  
  \ear
\end{equation}

  In this paper we only present the susbtitution method on
  the \\
  \textsf{pmrp}($A,b,c,z,x,h$; $(m, n), \textsf{cost})$ problem. We use the
  Fourier's trick (see eg. \cite{kn:Williams86}, \cite{kn:Juhel2025}) which
  consists in replacing $\max(z=c^{\intercal}x + c_{h}h)$ by ($0$): $\max(z)$,
  ($1$): $z \leq c^{\intercal}x + c_{h}h$ and ($2$): $c^{\intercal}x + c_{h}h \geq 0$. Then,
  \textsf{pmrp}($A,b,c,z,x,h$; $(m, n), \textsf{cost})$ is defined as:

  \begin{subequations}
    \begin{equation}
\max(z)
      \end{equation}
    such that:
    \begin{equation}
      \label{eqconevar}
(x,h) \in \overline{\real}^{n} \times \overline{\real},
    \end{equation}
    and
    \begin{equation}
      \label{eqleconeprincipal}
      \bar{l}
      z-\textsf{cost}(x,h) \leq 0 \\
      -\textsf{cost}(x,h) \leq 0 \\
      Ax -bh \leq \bzero.
      \ear
    \end{equation}
    with
    \begin{equation}
\textsf{cost}(x,h)= c^{\intercal}x + c_{h}h
      \end{equation}
    Where the homogenization variable is
    assumed to be such that:
    \begin{equation}
      \label{hneq0}
h \neq 0.
    \end{equation}
    And
    \begin{equation}
\mbox{$z$ and $h$ are not substituable variables}.
    \end{equation}
    Finally, among the variables of the problem $z, x_{1}, \ldots, x_{n}$ we have:
\begin{equation}
  \label{zprio}
  \mbox{$z$ has the highest priority because bounded by the cost function}.
  \end{equation}
    \end{subequations}

  \subsection{Organization of the paper}
  \label{subsecOrga}
  In section~\ref{secconeaveclepmrp} we define the geometrical object
  associated with the \textsf{pmrp}. The section~\ref{secImpResu} is devoted to the
  preliminary results: subfixed point theory and interval arithmetic. The section~\ref{secConeCipgelt}
  is devoted to the study of the bounds on both the cost variable $z$ and a variable $x_{j}$ of the problem. \\

  The most important section is section~\ref{sec-substitution}. This section~\ref{sec-substitution}
  deals with the criteria for the substitution. We introduce the concept of hierarchy $\preceq_{var}$
  between the variables of the LPP. The hierarchy is based on the reachability domain of
  the variables. The variables are ordered from the smallest domain to the
  greatest domain in the sense of the inclusion of sets. In fact the variables associated with a 
  smaller domain than $[0,+\infty]$ which is necessarily of the form $[a,+\infty]$ with $a >0$ contribute to validate the
  hypothesis $\textbf{H}$ (see (\ref{hypCAbnontrivial})) dealing with the non triviality of the
  cone $\mathcal{C}(A,b)$. 
  We also introduce and define a partition for the non null
  linear functions into $h$-bounded functions and $h$-unbounded functions. But even if the concepts
  are close to the similar concepts of $h$-bounded functions and $h$-unbounded functions developed
  for the maximization of a $(\max,+)$-linear programming problem \cite[Appendix A]{truffet:hal-05556396}
  there is no global hierarchy (see Main difficulty~\ref{difficile}). We introduce a weaker concept called
  {\em conditioned hierarchy} (see Proposition~\ref{propoIneqhBvshUHierarchy}) which is sufficient for
  solving the LPP. The function $\textsf{Nearest-to-zero-criterion-for}$ (see p. \pageref{nearestfct}) which computes
  the smallest interval $\tau$ among a set of given intervals is presented in this section and
  its main properties are studied in Proposition~\ref{propntzfctppte}. Also the two main theorems of this
  section are proved. The first one deals with
  the stopping condition of the \textsf{pmrp} (see Theorem~\ref{thmReachch.h}). The second theorem,
  ie. Theorem~\ref{thmNTZ}, deals with the theoretical criteria for substitution and the computation of
  substituable variables which satisfy these criteria. The computation of the criteria depends on
  several cases. These cases are based on the ``shape'' of the cost function: existence
  of variables with strictly positive cost, null cost or strictly negative cost
  (see \S~\ref{subsec:substitutionProcedure}). The computation (in $\mathcal{O}(mn^{2})$ complexity)
  of the new characteristics of the
  problem after the susbtitution of a substituable variable is also presented in \S~\ref{subsecNewcost+Newconepmrp}. \\
  
  In section~\ref{secFourierMotzkinvsSubstit} we discuss the main differences
  that exist between the Fourier-Motzkin elimination method and our substitution method
  which can be considered as a particular elimination method.

  Even if the concluding section~\ref{secConcl} studies the time complexity of the
  method with more details we already claim in this subsection that the time complexity for the \textsf{prmp} is about
  $\mathcal{O}(mn^{2})$ (for finding a variable for substitution repeated $n$ times)
  plus $\mathcal{O}(mn^{3})$ (applying procedure of \S~\ref{subsecNewcost+Newconepmrp} $n$ times)
  plus $\mathcal{O}(n)$ (solving a triangular system of $n$ equations by backward
  substitution).
  
  Finally, two numerical examples (positive and negative maximum) are given
  in section~\ref{exemplePedagogiques}.

\section{Polyhedral cone associated with the \textsf{pmrp}}
\label{secconeaveclepmrp}  
  The geometrical object associated with
  (\ref{eqconevar})-(\ref{eqleconeprincipal}) is the polyhedral cone denoted
  $\mathcal{C}(\overline{A})$ and defined as the following set in the $(z,x,h)$-system of variables.

\begin{subequations}
  \label{compactposs}
  \begin{equation}
    \label{compactpossineg}
\mathcal{C}(\overline{A}):=\{w \in \overline{\real}^{[|0,n+1|]}: \overline{A} w \leq \bzero\},
    \end{equation}
where the matrix $\overline{A} \in \real^{[|-1,m|],[|0,n+1|]}$ is defined by:
\begin{equation}
    \label{compactpossAw=zxh}
    (\overline{A},w):= \left(\bar{ccc}
     1 & -c^{\intercal} & -c_{h} \\
    0 &  -c^{\intercal} & -c_{h} \\
    \bzero & A & -b
    \ear \right), \left(\bar{c} z \\ x \\ h \ear \right).
\end{equation}
\end{subequations}

To the vector of the remaining variables $x$ we associate the set of the remaining variables
denoted $\x$. We have the following noticeable equivalence:

\begin{equation}
  \label{eqset-vect}
x_{j} \notin \x \Leftrightarrow \mbox{ $x_{j}=0$ in the vector of the remaining variables $x$}.
  \end{equation}

Two noticeable functions $\textsf{checknulvar}$ and $\textsf{setrowtozero}$ are associted with the cone
$\mathcal{C}(\overline{A})$.

\begin{equation}
  \label{setrowtozero0}
  \bar{lll}
  \textsf{setrowtozero}(\overline{A}) & := & \mbox{For $i=0$ to $m$ do} \\
\mbox{} & \mbox{} & \mbox{if $\overline{a}_{i,.} \leq \bzero^{\intercal}$ then $\overline{a}_{i,.}:=\bzero^{\intercal}$}.
  \ear
    \end{equation}

\begin{equation}
  \label{efchecknulvar0}
  \bar{lll}
  \textsf{checknulvar}(\overline{A}) & := & \mbox{For $i=0$ to $m$ do} \\
\mbox{} & \mbox{} &   \mbox{For $j=0$ to $n+1$ do} \\
\mbox{} & \mbox{} & \mbox{if $\overline{a}_{i,.} \geq \bzero^{\intercal}$ and $\overline{a}_{i,j} > 0$
  then $w_{j}:= 0$}.
\ear
\end{equation}

We adopt the following numerotation convention for the matrice $\overline{A}$

\begin{numerotation}
  \label{desnotations}
  The rows of the matrix $\overline{A}$ are
  numbered from $-1$ to $m$. And sometimes the columns of the matrix $\overline{A}$
  are labelled by the variables $z, x_{1}, \ldots, x_{n}, h$ in this order instead of the set $[|0,n+1|]$.
  \end{numerotation}

Let $\underline{w}$ be a lower bound (finite or not) of the vectors $w \in \mathcal{C}(\overline{A})$. We study
the following set function $s: w \mapsto \overline{A}([\underline{w}, w]) \cap \overline{\real}^{n+2}$ where
$[\underline{w}, w]:= \{\omega \in \overline{\real}^{n+2}: \underline{w} \leq \omega \leq w\}$ and
$\overline{A}([\underline{w}, w])$ denotes the image of the set $[\underline{w}, w]$ by the
linear application $\omega \mapsto \overline{A} \omega$. 

\begin{prpstn}
  \label{propUpperBdfirst}
  If $w$ has a finite upper bound $\overline{w}$ then we have:

  \begin{equation}
    \label{eqUp}
\sup_{\underline{w} \leq w \leq \overline{w}}s(w) = s(\overline{w}).
    \end{equation}
  Where $\sup$ denotes the supremum in the sense of the set inclusion $\subseteq$.

  And if $w$ has no finite upper bound then:
  \begin{equation}
    \label{eqlow}
\sup_{\underline{w} \leq w }s(w) = s(\underline{w}).
    \end{equation}
  
  \end{prpstn}
\proof To prove (\ref{eqUp}) we just have to note that $\forall
\underline{w} \leq w \leq w' \leq \overline{w}$ we have the following
set inclusion: $[\underline{w}, w] \subseteq [\underline{w}, w']$. And
because $\overline{A}(\cdot)$ is an application we have:
$\overline{A}([\underline{w}, w]) \subseteq \overline{A}([\underline{w},
  w']) \subseteq \overline{A}([\underline{w},\overline{w}])=s(\overline{w})$. \\

The proof of (\ref{eqlow}) is based on the following set inclusion:
$[\underline{w}, \boldsymbol{+\infty}] \subseteq [w, \boldsymbol{+\infty}], \forall \underline{w} \leq w$
and the application property.

\cqfd

\section{Preliminary results}
\label{secImpResu}

\subsection{Subfixed point problem}
\label{subfixedpointresu}
First, let us recall classical results dealing with the convergence
of the Neumann series of a square matrix.

\begin{rslt}[See e.g. \cite{kn:Meyer2000}, \cite{kn:Berman79}]
  \label{resNeumannSeries}
  Let $G$ be a $n \times n$-matrix. The Neumann series of $G$ is
  defined as $I + G + G^{2} + \cdots$. The following conditions
  are equivalent.

  \begin{subequations}
    \begin{equation}
      \mbox{The Neumann series of matrix $G$ converges.}
    \end{equation}

  \begin{equation}
\limk G^{k} = \bO_{n}.
    \end{equation}
    \begin{equation}
      \mbox{The spectral radius of matrix $G$ is $< 1$.}
    \end{equation}

  \end{subequations}
  In this case $(I-G)^{-1}$ exists and $(I-G)^{-1}=\sum_{i=0}^{\infty} G^{i}$.
  
  Moreover, we have:
  \begin{equation}
G \geq O \Leftrightarrow (I-G)^{-1} \geq O. 
  \end{equation}
$(I-G)^{-1}$ will be denoted $G^{*}$ when $(I-G)^{-1}$ and $G$ are nonnegative matrices.
  
  \end{rslt}

Let $G$ be a nonnegative $n \times n$-matrix, ie. $G \geq \bO_{n}$,
and $b$ a $n$-dimensional vector. We consider the following set:

\begin{equation}
  \label{eqSubpoitfix}
\mathcal{S}^{\leq}:= \{x \in \overline{\real}^{n}: x \leq G x + b\},
  \end{equation}

As a consequence of Result~\ref{resNeumannSeries} we derive in the next Theorem the necessary
and sufficient condition
which warranties the existence of the graetest element of the set $\mathcal{S}^{\leq}$.

\begin{thm}
  \label{theoFixpoint}
  The set $\mathcal{S}^{\leq}$ admits a greatest element $x^{sup}$ iff 
  the following condition
  \begin{equation}
\limk G^{k} = \bO_{n},
    \end{equation}
holds. And this elment have the following algebraic expression:

  \begin{equation}
x^{sup} = G^{*} b.
    \end{equation}
\end{thm}
\proof Because $G$ is a nonnegative matrix, by an easy induction we
have:

\[
x \leq G^{k}x + (\sum_{l=0}^{k-1}G^{l})b.
\]
And the Result~\ref{resNeumannSeries} ends the proof. \cqfd

We borrow from \cite{kn:Berman79} the following definitions.
s a square binary matrix that has exactly one entry of 1 in each row and each column with all other entries 0.
\begin{dfntn}[Permutation matrix]
  A square matrix $P=(p_{i,j})$ is a permutation matrix if it has $\{0,1\}$-valued entries
  and has exactly one $1$-valued entry in each row and each column.
  
\end{dfntn}

\begin{dfntn}[Monomial matrix]
  \label{defmonomialMat}
  A square matrix $M$ is a monomial matrix if $M$ has the form:
 
  \begin{equation}
M=DP
  \end{equation}
  where $D$ is a diagonal matrix and $P$ is a permutation matrix.
  \end{dfntn}

Monomial matrices have the following property which plays a crucial role in this
work.

\begin{prprt}[see eg. \cite{kn:DingRhee2014}]
  \label{MetM-1positive}
  A square matrix $M$ and its inverse $M^{-1}$ are nonnegative iff
  $M$ is a nonnegative invertible monomial matrix. No other kind
  of matrix has this property.
    \end{prprt}

Let $u$ be a $k$-dimensional vector and $\overline{u}$ be a $n-k$ dimensional
vector, $n > k \geq 1$. Let $M, N$ be two nonnegative $k \times k$-matrices such that $M$ is monomial and
invertible. Finally, let $R, R'$ be two nonnegative $k \times (n-k)$-matrices.

Let us define the following set:
\begin{equation}
  \label{eqdefSleq}
    \mathcal{S}^{\leq}(u, \overline{u}) :=
    \{(u, \overline{u}): M u + R \overline{u} \leq N u + R' \overline{u} \}.
 \end{equation}
  The following result is a simple consequence of
  Theorem~\ref{theoFixpoint} and Property~\ref{MetM-1positive}.

  \begin{coro}
  \label{coroMonomialFixpoint}
  The set $\mathcal{S}^{\leq}(u, \overline{u})$ has a greatest element, say $u^{sup}$,  iff
  \begin{equation}
    \label{limkM-1N=O}
  \limk (M^{-1} N)^{k} = \bO_{k}.
  \end{equation}
  This element are unique and defined as a linear function of
  $\overline{u}$ by:

  \begin{equation}
    \label{usupuinf}
u^{sup}:= U \overline{u},
    \end{equation}
with  
\begin{equation}
  \label{matU}
    U:= (M-N)^{-1} \; (R' - R).
  \end{equation}
\end{coro}
  \proof Because $M$ and $M^{-1}$ are nonnegative we have the following equivalences:

  \[
  M u + R \overline{u} \leq N u + R' \overline{u} \Leftrightarrow u
  \leq M^{-1} N u +
  M^{-1} (R' - R) \overline{u}. 
  \]

  Then, Theorem~\ref{theoFixpoint} applies with $G:= M^{-1} N$ and $b:=M^{-1} (R' - R) \overline{u}$.
  Then we just have to note that:
  \[
  \bar{lll}
  (M^{-1} N)^{*} M^{-1} & = & (I-M^{-1}N)^{-1}M^{-1} \\
  \mbox{} & = & (M (I-M^{-1}N))^{-1} \\
  \mbox{} & = & (M-N)^{-1}.
  \ear
  \]
And the corollary is proved. \cqfd

\subsection{Useful elementary results borrowed from interval arithmetic}
\label{subIntervalresu}
In this subsection we recall  basic results on interval arithmetic
\cite{kn:Moore79} needed for our purpose. The interval arithmetic is defined
on $\overline{\real}:=\real \cup \{-\infty, +\infty\}$ (recall). If $u
  \leq v \in \overline{\real}$ the interval $[u,v]$ is defined
  as the following set $[u,v]:=\{x: u \leq x \leq v\}$. Otherwise it is equal
  to $\emptyset$.



  We only present the following necessary results borrowed from
  interval arithmetic theory.

  \bit
  \item Equality of two intervals. $[u,v]=[u',v']$ if $u=u'$ and $v=v'$.

    \item Multiplication of an interval by a real.
      For all $u \leq v \in \overline{\real}$ and $\forall \alpha \in \real$,
      we have:

      \begin{equation}
        \label{eqMult}
        \alpha \; [u,v] = \left\{\bar{cc} \mbox{$[\alpha u, \alpha v]$} & \mbox{ if $\alpha \geq 0$} \\
          \mbox{$[\alpha v, \alpha u]$} & \mbox{ if $\alpha < 0$}
            \ear \right.
        \end{equation}

    \item Addition of two intervals. If $u \leq v \in \overline{\real}$ and
      $u' \leq v' \in \overline{\real}$ the addition of the
      intervals $[u,v]$ and $[u',v']$ denoted with the same symbol as
      the addition on $\real$, i.e. $+$, is defined by:

      \begin{equation}
        \label{eqPlus}
        [u,v] + [u',v'] = [u+u', v+v'].
      \end{equation}
    \item Linear functions of intervals. Let $k \geq 2$ and  $\alpha_{j}, j=1, \ldots, k$ be $k$
      elements of $\real$. Let $I_{j}:=  [u_{j},v_{j}]$, $j=1, \ldots, k$ be a familly
      of $k$ intervals. The linear function of intervals $I:= \sum_{j=1}^{k} \alpha_{j} \; I_{j}$ is also an interval
      $[\textsf{lo}, \textsf{up}]$ defined by:

      \begin{equation}
        \label{lincombinterval}
        \bar{lll}
        \textsf{lo} & = & \sum_{\{j: \alpha_{j} < 0\}} \alpha_{j} v_{j} + \sum_{\{j: \alpha_{j} \geq 0\}} \alpha_{j} u_{j}\\
        \textsf{up} & = & \sum_{\{j: \alpha_{j} < 0\}} \alpha_{j} u_{j} + \sum_{\{j: \alpha_{j} \geq 0\}} \alpha_{j} v_{j}.
        \ear
        \end{equation}

\eit
      
     Due to the Assumption~\ref{AsRlambda} of \S~\ref{introAllthematos} we will define
     and study the properties of the following set of the symmetric intervals.

      \begin{dfntn}
        \label{defIntervalsymm}
        We define the set of intervals $\mathcal{I}_{0}$ as:
        \begin{equation}
          \label{eqDefintervalsymm}
\mathcal{I}_{0}:= \{[-v,v], \; v \in \real_{+} \cup \{+\infty\}\}.
          \end{equation}
        \end{dfntn}

     Now we study the noticeable properties of the set of symmetric intervals
     $\mathcal{I}_{0}$ which will be useful for the $\lambda$-parametrized
     research of positive finite maximum to the LPP. Obviously we have
     
      \begin{prpstn}
        \label{propOrdreTotal}
        The set inclusion, $\subseteq$, is a total order on the set
        of intervals $\mathcal{I}_{0}$.
      \end{prpstn}
      
      And we also have the following result dealing with smallest and greatest element of a familly of
      symmetric intervals.
      
      \begin{prpstn}
        \label{propMinInterval}
        Let $F=\{I_{1}=[-\lambda_{1}, \lambda_{1}], \ldots , I_{k}=[-\lambda_{k}, \lambda_{k}]\}$ be a familly of $k$, $k \geq 1$,
        intervals of $\mathcal{I}_{0}$ such that $0 \leq \lambda_{1} \leq \ldots \leq \lambda_{k} \leq +\infty$. Then, $F$
        is totally ordered set which has at least one
        smallest element and one greatest element in the sense of the inclusion denoted
        $\textsf{Min}(F)$ and $\textsf{Max}(F)$ respectively. We have: 
\begin{subequations}
        \begin{equation}
          \label{eqMininterval}
\textsf{Min}(F) = [-\lambda_{1}, \lambda_{1}],
        \end{equation}
        and 
        \begin{equation}
          \label{eqMaxinterval}
\textsf{Max}(F) = [-\lambda_{k}, \lambda_{k}]
        \end{equation}
\end{subequations}

Moreover we also note that:

\begin{subequations}
        \begin{equation}
          \label{eqMininterval1}
\textsf{Min}(F) = \cap_{i=1}^{k} [-\lambda_{i}, \lambda_{i}],
        \end{equation}
        and 
        \begin{equation}
          \label{eqMaxinterval1}
\textsf{Max}(F) = \cup_{i=1}^{k} [-\lambda_{i}, \lambda_{i}]
        \end{equation}
\end{subequations}

\end{prpstn}
      
      \begin{prpstn}
        \label{propCalculBornesInterv}
  Let us consider the linear function $f(x):= \alpha x= \sum_{j=1}^{n} \alpha_{j} x_{j}$. And assume that
  $\forall j: x_{j} \in [-\lambda, \lambda]$, $\lambda \in \real_{+}$. Then, the image
  of $f$ is the interval $I(f):=[-v,v]$ with $v:=  |\alpha|_{1} \lambda$, 
  where $|\alpha|_{1} = \sum_{j=1}^{n} |\alpha_{j}|$.
  \end{prpstn}

      \section{Studying the cone \\
        $\mathcal{C}(i):=\{\overline{a}_{-1,.} w \leq 0, \; \overline{a}_{i,.} w \leq 0\}$}
\label{secConeCipgelt}

In this part the cone $\mathcal{C}(i)$ is associated with the intial system of inequalities
which defines the polyhedron $\mathcal{P}(A,b)$. Thus, we do not consider row $0$ of the
cone $\mathcal{C}(\overline{A})$ which deals with the positivity requirement for the cost
function.

We consider a row $i \in [|1,m|]$ and as for the usual substitution to the hyperplan defined by:
$\{w: \overline{a}_{i,.} w =0\}$, if
$\overline{a}_{i,j} \neq 0$ for a $j \in [|1,n|]$, we define the $(x,h)$-linear function $f_{ij}(x,h)$, called
{\em $f_{ij}$-function}, by:

    \begin{subequations}
      \begin{equation}
        \label{deffij}
f_{ij}(x,h):= \ell_{ij}(x) + r_{ij} h, 
    \end{equation}
    with $\ell^{+}_{ij}$ is the following linear function:
    \begin{equation}
      \label{deflij}
\ell_{ij}(x):= v^{\intercal}_{ij} x,
    \end{equation}
    where $v_{ij}$ denotes the  $n$-dimensional row vector such that:
    \begin{equation}
      \label{eqdefuij}
      \forall j'=1, \ldots, n: \; v_{ij,j'}=\left\{\bar{cc}
      -\overline{a}_{i,j}^{-1} \overline{a}_{i,j'} & \mbox{if $j' \neq j$} \\
      0 & \mbox{if $j'=j$}
      \ear \right.
    \end{equation}
   and the real $r_{ij}$ is defined by:
    \begin{equation}
      \label{eqdefrij}
r_{ij}:= \overline{a}_{i,j}^{-1}  b_{i}.
      \end{equation}
    \end{subequations}

\subsection{Existing condition and properties of the upper bound on $z$ and $x_{j}$ w.r.t $\mathcal{C}(i)$}
\label{existbdonz+xj}
Recall that $\x$ denotes the set of the remaining variables in the substitution process.
In this part we assume that the set $\{x_{k} \in \x: c_{k} \geq 0\}$ is not empty and
contains the variable $x_{j}$ for a $j \in [|1,n|]$. We consider the half space 
$\{w: \overline{a}_{i,.} w  \leq 0\}$ and we assume that $\overline{a}_{i,j} >0$. Thus, to the inequality
which characterizes this half space
we associate the equivalent inequality defined by:

 \begin{equation}
I_{\leq}(x_{j},i):=\{x_{j} \leq f^{\leq}_{ij}(x,h)\}.
  \end{equation}
 Where $f^{\leq}_{ij}(x,h)$ is a $f_{ij}$-function defined by ((\ref{deffij})-(\ref{eqdefrij}), with $\overline{a}_{i,j} >0$).

 One can also define the following function by replacing $x_{j}$ with $f^{\leq}_{ij}(x,h)$ in the
 cost function $c^{\intercal}x + c_{h}h$:

      \begin{equation}
        \label{defz+xj}
f_{z,ij}(x,h):=  \sum_{j' \neq j} c_{j'} x_{j'} + c_{j} f^{\leq}_{ij}(x,h)  + c_{h} h. 
        \end{equation}
And replacing $f^{\leq}_{ij}(x,h)$ in the above expression we obtain:

     \begin{equation}
 \label{fzijposs}
f_{z,ij}(x,h)  =   \sum_{j' \neq j} (c_{j'} + c_{j} v_{ij,j'}) x_{j'} + (c_{h} + c_{j} r_{ij}) h,
\end{equation} 

     Clearly, the function $f_{z,ij}(x,h)$ is the cost function where the variable
     $x_{j}$ has been substituted by the $f_{ij}$-function $f^{\leq}_{ij}$. And
     because $c_{j} \geq 0$ we have:

     \[
     \left(\bar{c}
z \\ x_{j} 
     \ear\right) \leq  \left(\bar{c}
f_{z,ij}(x,h) \\ f^{\leq}_{ij}(x,h) 
     \ear\right)
     \]

In the next Theorem we prove that the $2$-dimensional vector $\left(\bar{c} f_{z,ij}(x,h)
\\  f^{\leq}_{ij}(x,h) \ear \right)$ is a greatest element of $\mathcal{C}(i)$.

\begin{thm}
  \label{propVarsubstitposs}
  If the variable $x_{j} \in \x$ is such that $c_{j} \geq 0$, then the cone $\mathcal{C}(i)$ admits
  a greatest element 
  $u^{\max}(z,x_{j},i):=\left(\bar{c} f_{z,ij}(x,h)
  \\  f^{\leq}_{ij}(x,h) \ear \right)$ iff $\overline{a}_{i,j}=a_{i,j} > 0$.

  Defining:
  \begin{equation}
    \label{defuubareleq}
    u:= \left(\bar{c} z \\ x_{j} \ear \right), \; \overline{u}:=\left(\bar{c}(x_{k})_{x_{k} \in x, k \neq j} \\
    h \ear \right), 
\end{equation}
    and
  noticing that $u^{\max}(z,x_{j},i)$ only depends on $\overline{u}$, we have:

  \[
\mathcal{C}(i) =\{u, \overline{u} : u \leq u^{\max}(z,x_{j},i) \}. 
  \]

\end{thm}
\proof To prove the result we use the two-sided form in the $(z,x,h)$-system of variables
  of the cone
  $\mathcal{C}(\overline{A})$ (see (\ref{compactpossineg})-(\ref{compactpossAw=zxh})) associated
  with \textsf{pmrp}.
This two-sided form is obtained after
some easy manipulations of the form
`` $\rho^{+}- \rho^{-} \leq 0 \Leftrightarrow \rho^{+} \leq \rho^{-}$ `` with
$\rho^{+}, \rho^{-}$ which are nonnegative vectors and/or matrices with the same dimensions
as the original unsigned vectors $b,c$ and matrix $A$. We have:

\begin{subequations}
  \label{2sidedposs}
  \begin{equation}
    \label{2sidedpossineg}
\overline{A}^{\; +} w \leq \overline{A}^{ \; -} w,
    \end{equation}
  where the matrices $\overline{A}^{\; +} \in \real_{+}^{[|-1,m|],[|0,n+1|]}$ and $\overline{A}^{\; -} \in \real_{+}^{[|-1,m|],[|0,n+1|]}$
  are defined by:
  \begin{equation}
    \label{2sidedpossA+A-wzxh}
(\overline{A}^{\; +}, \overline{A}^{\; -},w):= \left(\bar{ccc}
     1 & c^{-} & c_{h}^{-} \\
    0 & c^{-} & c_{h}^{-} \\
    \bzero & A^{+} & b^{-}
    \ear \right),
    \left(\bar{ccc}
     0 & c^{+} & c_{h}^{+} \\
    0 & c^{+} & c_{h}^{+} \\
    \bzero & A^{-} & b^{+}
    \ear \right), \left(\bar{c} z \\ x \\ h \ear \right).
  \end{equation}
\end{subequations}
  
  Then, we consider the
  set $\mathcal{S}^{\leq}(u, \overline{u})$ defined by (\ref{eqdefSleq}) where the
  vector are defined by (\ref{defuubareleq}).

The matrices $M, N, R, R'$ involved in the definition of the set $\mathcal{S}^{\leq}(u, \overline{u})$ are
the following submatrices of the matrices $\overline{A}^{ \; +}$ and $\overline{A}^{ \; -}$. First, the $2 \times 2$-matrices
$M$ and $N$ are respectively defined by:
\begin{equation}
  \label{defMposs}
M := \left(\bar{cc} \overline{a}^{\;+}_{-1, 0} &  \overline{a}^{\;+}_{-1,j} \\
\overline{a}^{\;+}_{i,0} & \overline{a}^{\;+}_{i,j} 
\ear \right),
\end{equation}
and
\begin{equation}
  \label{defNposs}
N := \left(\bar{cc} \overline{a}^{\;-}_{-1, 0} &  \overline{a}^{\;-}_{-1,j} \\
\overline{a}^{\;-}_{i,0} & \overline{a}^{\;-}_{i,j}. 
\ear \right)
\end{equation}
Second, the $2 \times n$-matrices $R$ and $R'$ are respectively defined by:
\begin{equation}
  \label{defRlposs}
R:= \left(\bar{ccccccc}
\overline{a}^{\;+}_{-1, 1} & \cdots & \overline{a}^{\;+}_{-1, j-1} & \overline{a}^{\;+}_{-1, j+1} & \cdots & \overline{a}^{\;+}_{-1, n} & \overline{a}^{\;+}_{-1, n+1} \\
\overline{a}^{\;+}_{i, 1} & \cdots & \overline{a}^{\;+}_{i, j-1} & \overline{a}^{\;+}_{i, j+1} & \cdots &\overline{a}^{\;+}_{i, n} & \overline{a}^{\;+}_{i, n+1} \\
\ear \right),
\end{equation}
and
\begin{equation}
  \label{defRrposs}
R':= \left(\bar{ccccccc}
\overline{a}^{\;-}_{-1, 1} & \cdots & \overline{a}^{\;-}_{-1, j-1} & \overline{a}^{\;-}_{-1, j+1} & \cdots & \overline{a}^{\;-}_{-1, n} & \overline{a}^{\;-}_{-1, n+1} \\
\overline{a}^{\;-}_{i, 1} & \cdots & \overline{a}^{\;-}_{i, j-1} & \overline{a}^{\;-}_{i, j+1} & \cdots & \overline{a}^{\; -}_{i, n} & \overline{a}^{\;-}_{i, n+1} \\
\ear \right).
\end{equation}

By definition of the matrix $ \overline{A}^{ \; +}$ in (\ref{2sidedpossA+A-wzxh}) we
have:
\[
M = \left(\bar{cc} 1 &   c^{-}_{j} \\
0 & \overline{a}_{i,j}
\ear \right)
\]
because $I_{\leq}(x_{j},i)$ is assumed to exist thus $\overline{a}_{i,j} > 0$ and
then $\overline{a}^{\;+}_{i,j}= \overline{a}_{i,j}$.

By definition of the matrix  $\overline{A}^{ \; -}$ in (\ref{2sidedpossA+A-wzxh}) and because
$\overline{a}_{i,j} > 0$ which implies $\overline{a}^{\;-}_{i,j}.=0$, by construction, we
have:
\[
N= \left(\bar{cc} 0 &  c^{+}_{j} \\
0 & 0 
\ear \right).
\]
Because the only matrices that verify $A$ and $A^{-1}$ are nonnegative matrices
are the monomial matrices (see Property~\ref{MetM-1positive}) we necessary
have $c^{-}_{j}=0$ which is equivalent to $c_{j}=c^{+}_{j} \geq 0$. In such
case $M$ is invertible nonnegative diagonal matrix which inverse
is the following nonnegative diagonal matrix:
\[
M^{-1}=  \left(\bar{cc} 1 &   0 \\
0 & \overline{a}_{i,j}^{-1}
\ear \right).
\]

Moreover, we note that $M^{-1} N= N$ and $N^{2}=\bO_{2}$ thus the
necessary and sufficient condition $\limk (M^{-1} N)^{k} = \bO_{2}$
(see (\ref{limkM-1N=O})) is obviously true. Thus, Corollary~\ref{coroMonomialFixpoint}
applies and the greatest element of the set $\mathcal{S}^{\leq}(u, \overline{u})$
exists. Specifying to our case the greatest element, say $u^{sup}$ defined
by $U \overline{u}$ (see (\ref{usupuinf})) where the matrix $U$ is defined
by (\ref{matU}) we have the fundamental inequality:

\begin{equation}
  \label{eqfondamentaleposs}
  u= \left(\bar{c} z \\ x_{j} \ear \right) \leq \left(\bar{c} f_{z,ij}(x,h)
    \\  f^{\leq}_{ij}(x,h) \ear \right).
\end{equation}

And the proof is now achieved. 

\cqfd

Let us begin by enumerating the main results dealing with
inequalities $I_{\leq}(x_{j},i):=\{x_{j} \leq f^{\leq}_{ij}(x,h)\}$ in the next Proposition.
  
 \begin{prpstn}
   \label{resIleq}
  The obvious results are the following ones.

   \bit 
 \item For all $i=1,\ldots m$ and for all $j=1, \ldots , n$,
   $I_{\leq}(x_{j},i) = \emptyset$ iff $\overline{a}_{i,j} \leq 0$.

 \item For all $j=1, \ldots , n$ we define the set $I_{\leq}(x_{j})$ as follows:

   \begin{equation}
     \label{Ileqxj}
     I_{\leq}(x_{j}):= \cup_{i \in [|1,m|]}\{I_{\leq}(x_{j},i)\}.
     \end{equation}

   And we have: $I_{\leq}(x_{j})= \emptyset$   iff $\overline{A}_{[|1,m|]j} \leq \bzero$. In this case it is clear
   that $x_{j}$ is unbounded as a result of Proposition~\ref{propLambdaParamFausse}
   of \S~\ref{introAllthematos}. 
   
\item For all $i=1, \ldots ,m$, we define the set $I_{\leq}(i)$ as follows:

  \begin{equation}
     \label{Ileqi}
     I_{\leq}(i):= \cup_{x_{j} \in \x} \{I_{\leq}(x_{j},i)\}.
     \end{equation}
Then, we have $I_{\leq}(i)= \emptyset$
   iff  $\overline{a}_{i,.} \leq \bzero^{\intercal}$.
 \item Finally, defining the set $I_{\leq}(\x)$ by:

   \begin{equation}
     \label{Ileqtout}
     I_{\leq}(\x):= \cup_{i \in [|1,m|]}\cup_{x_{j} \in \x} \{I_{\leq}(x_{j},i)\}.
     \end{equation}
   We have $I_{\leq}(\x) = \emptyset$ iff $\overline{A}_{[|1,m|] \x} \leq \bO_{m,\textsf{n}(\x)}$.
     \eit
 \end{prpstn}

 We define the following important property we called \textsf{positivity property}
 of $(x,h)$-linear function.

 \begin{dfntn}[Positivity property]
 \label{defpositivproperty} 
     A $(x,h)$-linear function $f$ has the \textsf{positivity property} if
     the following implication is true.
   \begin{equation}   
    \label{fijpositivetoutletemps}
 \forall (x,h) ((x,h) \in \real_{+}^{n+1} \Rightarrow f(x,h) \geq 0).
  \end{equation}

\end{dfntn}
 
And we have the following sign property for the $f^{\leq}_{ij}$ functions.
 
 \begin{prpstn}[Positivity of the $f^{\leq}_{ij}$ functions]
   \label{propfijpositive}
  If the set $I_{\leq}(\x) \neq \emptyset$ then all
  $(x,h)$-functions $f^{\leq}_{ij}$ have the  \textsf{positivity property}
  (\ref{fijpositivetoutletemps}).
\end{prpstn}

 \proof
   We just have to note that if $(x,h) \in \real_{+}^{n+1}$ then necessarily $ 0 \leq x_{j}$ and
   by transitivity of $\leq$ the signed property is proved. \cqfd

\subsection{Existing condition and properties of the upper bound on $z$ and $-x_{j}$ w.r.t $\mathcal{C}(i)$}
\label{existbdonz-xj}
Recall once again that $\x$ denotes the set of the remaining variables in the substitution process.
And we only consider the rows $i \in [|1,m|]$ associated with the original system of inequalities: $A x \leq b$.
In this part we assume that the set $\{x_{k} \in \x: c_{k} \leq 0\}$ is not empty and
contains the variable $x_{j}$ for a $j \in [|1,n|]$. We consider the half space 
$\{w: \overline{a}_{i,.} w  \leq 0\}$ and we assume that $\overline{a}_{i,j} < 0$. Thus, to the inequality
which characterizes this half space
we associate the equivalent inequality defined by:

\begin{equation}
I_{\geq}(x_{j},i):=\{x_{j} \geq f^{\geq}_{ij}(x,h)\}.
\end{equation}

Where $f^{\geq}_{ij}$ is a $f_{ij}$-function defined by ((\ref{deffij})-(\ref{eqdefrij}) with $\overline{a}_{i,j} <0$).

Because $c_{j} \leq 0$ the inequality $x_{j} \geq f^{\geq}_{ij}(x,h)$ implies $(-c_{j}) (-x_{j}) \leq -f^{\geq}_{ij}(x,h)$
and because $f^{\geq}_{ij}$ is $(x,h)$-linear the latter inequality is equivalent to: $(-c_{j}) (-x_{j}) \leq f^{\geq}_{ij}(-x,-h)$.

We then define the following function associated with the cost function expressed in the $(-x,-h)$-sytem
as $-c^{\intercal} (-x) -c_{h} (-h)$:

      \begin{equation}
        \label{defz-xj}
g_{z,ij}(-x,-h):=  \sum_{j' \neq j} -c_{j'} (-x_{j'}) - c_{j} f^{\geq}_{ij}(-x,-h) - c_{h} (-h). 
        \end{equation}
      Replacing $f^{\geq}_{ij}(-x,-h)$ in the above expression we obtain:
      
      \begin{equation}
        \label{gzijposs}
g_{z,ij}(-x,-h) = \sum_{j' \neq j} (-c_{j'} - c_{j} v_{ij,j'}) (-x_{j'}) + (-c_{h} - c_{j} r_{ij}) (-h). 
        \end{equation}
      Under this form we have the following inequality which holds:
      \[
      \left(\bar{c}
      z \\
      -x_{j}
      \ear \right) \leq \left(\bar{c}
      g_{z,ij}(-x,-h) \\
      f^{\geq}_{ij}(-x,-h)
      \ear \right).
      \]

And in the next theorem we prove that $\left(\bar{c}
      g_{z,ij}(-x,-h) \\
      f^{\geq}_{ij}(-x,-h)
      \ear \right)$
      is a greatest element of $\mathcal{C}(i)$. 
      
\begin{thm}
  \label{propVarsubstitposs2}
  If the variable $x_{j} \in x$ is such that $c_{j} \leq 0$, then the cone $\mathcal{C}(i)$ admits
  a greatest element 
  $u^{grt}(z,-x_{j},i):=\left(\bar{c} g_{z,ij}(-x,-h)
  \\  f^{\geq}_{ij}(-x,-h) \ear \right)$ iff $\overline{a}_{i,j} < 0$.

  Defining:
  \begin{equation}
    u:= \left(\bar{c} z \\ -x_{j} \ear \right), \; \overline{u}:=\left(\bar{c} (-x_{k})_{x_{k} \in \x, k \neq j\}} \\
    -h \ear \right),
    \end{equation}
  and noticing that $u^{grt}(z,-x_{j},i)$ only depends on $\overline{u}$, we have:

  \[
\mathcal{C}(i) =\{u, \overline{u} : u \leq u^{grt}(z,-x_{j},i) \}. 
  \]
\end{thm}
\proof The proof uses same kind of arguments than the proof of Theorem~\ref{propVarsubstitposs} and is
not reproduced in the sequel. We only mention that we use the other two-sided form in the $(z,-x,-h)$-system of variables
  of the cone
  $\mathcal{C}(\overline{A})$ (see (\ref{compactpossineg})) associated with \textsf{pmrp}.
Once again this two-sided form is obtained after
some easy manipulations of the form
`` $\rho^{+}- \rho^{-} \leq 0 \Leftrightarrow \rho^{+} \leq \rho^{-}$ `` with
$\rho^{+}, \rho^{-}$ which are nonnegative vectors and/or matrices with the same dimensions
as the original unsigned vectors $b,c$ and matrix $A$. We have:

\begin{subequations}
  \label{2sidedpossother}
  \begin{equation}
    \label{2sidedpossinegother}
\overline{A}^{\; +} w \leq \overline{A}^{ \; -} w,
    \end{equation}
  where the matrices $\overline{A}^{\; +} \in \real_{+}^{[|-1,m|],[|0,n+1|]}$ and $\overline{A}^{\; -} \in \real_{+}^{[|-1,m|],[|0,n+1|]}$
  are defined by:

  \begin{equation}
    \label{2sidedpossA+A-wz-xh}
(\overline{A}^{\; +}, \overline{A}^{\; -},w):=  \left(\bar{ccc}
     1 & c^{+} & c_{h}^{+} \\
    0 & c^{+} & c_{h}^{+} \\
    \bzero & A^{-} & b^{+}
    \ear \right),
     \left(\bar{ccc}
     0 & c^{-} & c_{h}^{-} \\
    0 & c^{-} & c_{h}^{-} \\
    \bzero & A^{+} & b^{-}
    \ear \right),  \left(\bar{c} z \\ -x \\ -h \ear \right).
  \end{equation}
  
\end{subequations}

\cqfd

      Let us begin by enumerating the noticeable results in the next Proposition.
\begin{prpstn}
    \label{resIgeq}
  The obvious results are:

   \bit 
 \item For all $i=1,\ldots m$ and for all $j=1, \ldots , n$,
   $I_{\geq}(x_{j},i) = \emptyset$ iff $\overline{a}_{i,j} \geq 0$.
   
 \item For all $j=1, \ldots , n$ we define the set $I_{\geq}(x_{j})$ as follows:

   \begin{equation}
     \label{Igeqxj}
     I_{\geq}(x_{j}):= \cup_{i \in [|1,m|]} I_{\geq}(x_{j},i).
     \end{equation}

And we have: $I_{\geq}(x_{j})= \emptyset$   iff $\overline{A}_{[|1,m|]j} \geq \bzero$.
\item For all $i=0, \ldots ,m$, we define the set $I_{\geq}(i)$ as follows:

  \begin{equation}
     \label{Igeqi}
     I_{\geq}(i):= \cup_{x_{j} \in \x} I_{\geq}(x_{j},i).
     \end{equation}
Then, we have $I_{\geq}(i)= \emptyset$
   iff  $\overline{a}_{i,.} \geq \bzero^{\intercal}$.
 \item Finally, defining the set $I_{\geq}(\x)$ by:

   \begin{equation}
     \label{Igeqtout}
     I_{\geq}(\x):= \cup_{i \in [|1,m|]}\cup_{x_{j} \in \x} I_{\geq}(x_{j},i).
     \end{equation}
   We have $I_{\geq}(\x) = \emptyset$ iff $\overline{A}_{[|1,m|] \x} \geq \bO_{m, \textsf{n}(\x)}$.
     \eit

\end{prpstn}

\begin{prpstn}[Sign study of the $f^{\geq}_{ij}$ functions]
  \label{propsigngij}
  The $f^{\geq}_{ij}$ functions have the \textsf{positive property} (\ref{fijpositivetoutletemps})
  iff 
\begin{equation}
  \label{Igeqpertinente}
v_{ij} \geq \bzero \mbox{ and } r_{ij} \geq  0.
\end{equation}

\end{prpstn}
\proof

   Let us consider any  $I_{\geq}(x_{j},i) \in I_{\geq}(\x)$. Recall that it means:
  
 \[
x_{j} \geq  f_{ij}(x,h)= v^{\intercal}_{ij} x + r_{ij}h = \sum_{j' \neq j} v_{ij,j'} x_{j'} + r_{ij} h
\]

For the proof we omit the superscripts $(\cdot)^{\leq}$ and $(\cdot)^{\geq}$ for the $f_{ij}$-functions
which are easily deduced by the context.

In the sequel we will need the following equality relations.

\begin{equation}
  \label{identiteUijj'j''}
\bar{lll}
-v_{ij,j'}^{-1}  v_{ij,j''}& = &-(-\overline{a}_{i,j'}^{-1}\overline{a}_{i,j'})^{-1} (-\overline{a}_{i,j}^{-1} \overline{a}_{i,j''}) \\
\mbox{} & = & -\overline{a}_{i,j'}^{-1}\overline{a}_{i,j''} \\
\mbox{} & = & v_{ij',j''}.
\ear
\end{equation}
And
\begin{equation}
   \label{identiteRijj'}
\bar{lll}
-v_{ij,j'}^{-1} r_{ij} & = & -(-\overline{a}_{i,j'}^{-1}\overline{a}_{i,j'})^{-1} (\overline{a}_{i,j}^{-1} b_{i}) \\
\mbox{} & = & \overline{a}_{i,j'}^{-1} b_{i} \\
\mbox{} & = & r_{i,j'}
\ear
\end{equation}

The case $v_{ij} \geq \bzero$ and  $r_{ij} \geq 0$ is obvious. In the sequel we only discuss
on the sign of the component of the vector $v_{ij}$. And we have to treat the following cases.

\bit
\item Case $1$. $\exists j' \neq j$ such that $v_{ij,j'} > 0$. \\
  \noi
The positivity of the linear function $f_{ij}$ we aim to prove is expressed
as follows:
\begin{equation}
  [Q]: x_{j'} \geq -v_{ij,j'}^{-1} \sum_{j'' \neq j,j'} v_{ij,j''} x_{j''}
  -v_{ij,j'}^{-1} r_{ij'}
\end{equation}
Recalling the following relations: $-v_{ij,j'}^{-1}  v_{ij,j''}= v_{ij',j''}$ (see (\ref{identiteUijj'j''}))
and $-v_{ij,j'}^{-1} r_{ij}=r_{i,j'}$ (see (\ref{identiteRijj'})) the condition $[Q]$ can be rewritten
as follows:
\begin{equation}
[Q]: x_{j'} \geq f:= \sum_{j'' \neq j,j'} v_{ij',j''} x_{j''} + r_{ij'}
\end{equation}
By assumptions we are in the case
where $v_{ij,j'} =-\overline{a}_{i,j}^{-1} \overline{a}_{i,j'} > 0$ and $\overline{a}_{i,j} < 0$. The sign rule implies
that $\overline{a}_{i,j'} > 0$. Thus, the inequality $I_{\leq}(x_{j'},i) \in I_{\leq}$ exists and is
written as:

\begin{equation}
[A']: x_{j'} \leq f_{ij'}(x,h):=\sum_{j" \neq j'} v_{ij',j''} x_{j''} + r_{ij'} h.
v\end{equation}
We have the equality: $f_{ij'}(x,h)= v_{ij',j} x_{j} + f$. We also have
$\forall (x,h) \; ((x,h) \in \real_{+}^{n+1} \Rightarrow  v_{ij',j} x_{j} = v_{ij,j'}^{-1} x_{j} \geq 0)$.
Then, we have the following inequalities:
\[
\bar{lll}
f & \leq & f_{ij'}(x,h) \\
x_{j'} & \leq & f_{ij'}(x,h) 
\ear
\]
which clearly do not imply $x_{j'} \geq f$. Thus $[A']$ is true but
$[A'] \Rightarrow [Q]$ is false. Thus, $[Q]$ is false.

\item $\forall j' \neq j$: $v_{ij,j'} \leq 0$. We have to treat the following two subcases.

  \bit

\item $v_{ij}= \bzero$. In this case the inequality $I_{\geq}(x_{j},i)$ is reduced to:
  \begin{equation}
x_{j} \geq r_{ij} h.
    \end{equation}
Assuming $(x,h) \in \real_{+}^{n+1}$ the linear function is nonnegative iff $r_{ij} \geq 0$, obviously.

\item $v_{ij} \neq \bzero$. Thus, $\exists j' \neq j$ such that $v_{ij,j'} < 0$. The positivity
  condition we aim to prove to be true is then expressed as $[Q'] :x_{j'} \leq f$. The strict inequalities $\overline{a}_{i,j} < 0$ and
  $v_{ij,j'} < 0$ imply $\overline{a}_{i,j'} < 0$. It means that $I_{\geq}(x_{j'},i) \in I_{\geq}(\x)$ exists. In other words
  the following inequality $[A'']: x_{j'} \geq f_{ij'}(x,h)$ is true.
  The relation $f_{ij'}(x,h) =  v_{ij',j} x_{j} + f$ still holds
  with $v_{ij',j}= v_{ij,j'}^{-1} < 0$. Thus, $(x,h) \in \real_{+}^{n+1} \Rightarrow f_{ij'}(x,h) \leq f$. Then, we have the
  following situation: $[A'']$ is true and $[A''] \Rightarrow [Q']$ is false. Thus, $[Q']$ is false.
  \eit
  \eit
  Except the obvious case ($v_{ij} \geq \bzero$ and $r_{ij} \geq 0$)
  the positivity property of the linear function $f^{\geq}_{ij}$ is false and the result is now proved.

 \cqfd

From the result of Proposition~\ref{propsigngij} we will only consider the {\em strictly positive}
$(x,h)$-linear lower bounding functions denoted $f^{\geq,+}_{ij}$ which are defined as follows.

\begin{dfntn}[Strictly positive function]
  \label{defstrictposs}
  
A function denoted $f^{\geq,+}_{ij}$ is said to be strictly positive if the following two conditions hold
  
\begin{subequations}
  \begin{equation}
    \label{inequalitystrictpositiv}
x_{j} \geq f^{\geq,+}_{ij}(x,h)= v^{\intercal}_{ij} x + r_{ij}h,
  \end{equation}
 with
  \begin{equation}
  \label{gijstrictpositive}
v_{ij} \geq \bzero \mbox{ and }  r_{ij} > 0.
\end{equation} 
\end{subequations}

An inequality satisfying (\ref{inequalitystrictpositiv})-(\ref{gijstrictpositive}) will be denoted
$I^{+}_{\geq}(x_{j},i)$. And the set of all such inequalities is denoted $I^{+}_{\geq}(\x)$. 
\end{dfntn}

An important consequence is that under the hypothesis $(\textbf{H})$ if
$I^{+}_{\geq}(x_{j},i) \neq \emptyset$ then the variable $x_{j}$
has a strict positive lower bound. In other words the domain of $x_{j}$ is strictly included into
$[0, +\infty]$. This, will induce a dominance relation between the remaining variables of $x$ developed
in \S~\ref{sec-substitution}

\section{How to choose the next variable to be substituted}
\label{sec-substitution}

\subsection{Classification of variables and linear functions}
\label{subHierarchy}
In this part we exhibit hierarchy between the variables of the LPP and
a partition of the non-null $(x,h)$-linear functions associated with inequalities in
$I_{\leq}$ or $I_{\geq}$. The partition/hierarchy is based on the value domain of variables and functions from the most
constrained domain to the less constrained domain. The most important variables/linear
functions are the ones which have a positive domain interval strictly included
in the standard reachability interval, say $[0,+\infty]$. Let us stress that no global hierarchy
is possible between linear functions (see Main difficulty~\ref{difficile}). This is the main
difference with the maximization of a $(\max,+)$-linear programming problem \cite{truffet:hal-05556396}. \\

For linear functions involved in the LPP we define the notion of $h$-boundedness as follows.

\begin{dfntn}[$h$-bounded $(x,h)$-linear function]
  \label{defhboundedfcts}
  A $(x,h)$-linear function $f:(x,h) \mapsto \alpha^{\intercal} x + \beta h$ with
  $\beta >0$ is $h$-bounded if

 \begin{equation}
\forall x \; (x \geq \bzero \Rightarrow f(x,h) \leq \beta h).
 \end{equation}
 
v\end{dfntn}

The set of $h$-bounded $(x,h)$-linear functions is denoted $h\B$
and the set
of $(x,h)$-linear functions which are not $h$-bounded $(x,h)$-linear functions is denoted
$h\U$. The elements of the set $h\U$ are called $h$-unbounded functions. It is clear that:

\begin{equation}
  \mbox{$f: (x,h) \mapsto \alpha^{\intercal} x + \beta h \in h\B$} \Leftrightarrow \alpha \leq \bzero, \beta > 0.
\end{equation}

In the next Proposition we provide the following noticeable result.

\begin{prpstn}[Main property of the elements of $h\B$]
  \label{propoSetincluhBup}
Let $f_{1}(x,h)=\alpha^{\intercal} x + \beta h$ and $f_{2}(x,h)=\gamma^{\intercal} x + \delta h$ be two
non-null $h$-bounded $(x,h)$-linear functions. Then, we have the following implication:

\begin{equation}
  0 < \beta \leq \delta \mbox{ and } 0 < h \Rightarrow \forall x \geq \bzero \; \exists x' \geq \bzero: f_{1}(x,h) \leq f_{2}(x',h). 
  \end{equation}
  \end{prpstn}
\proof. Take $x'=\bzero$. \cqfd \\

\begin{prpstn}[Main property of strictly positive functions]
  \label{propoSetincluStrictPosFct}
Let $f_{1}(x,h)=\alpha^{\intercal} x + \beta h$ and $f_{2}(x,h)=\gamma^{\intercal} x + \delta h$ be two
non-null strictly positive functions. Then, we have the following implication:

\begin{equation}
  0 < \beta \leq \delta \mbox{ and } 0 < h \Rightarrow \forall x \geq \bzero \; \exists x' \geq \bzero: f_{1}(x,h) \leq f_{2}(x',h). 
  \end{equation}
  \end{prpstn}
\proof. Take $x'=\bzero$. \cqfd \\

According to Propositions~\ref{propoSetincluhBup} and ~\ref{propoSetincluStrictPosFct} we define the
following $\lambda$-parametrized
domains for the vector of the remaining variables $x$ and the homogenization variable $h$
associated with the $(x,h)$-linear functions as follows.

\bit
\item For the linear functions in $h\B$ and strictly positive functions: 
  
  \begin{equation}
    \label{eqRB}
R^{\B}_{\lambda}(x,h):= \times_{x_{j} \in \x} [0,0] \times [-\lambda, \lambda].
  \end{equation}

\item And for the linear functions in $h\U$:
  
  \begin{equation}
    \label{eqRU}
R^{\U}_{\lambda}(x,h):= \times_{x_{j} \in \x} [-\lambda, \lambda] \times [-\lambda, \lambda].
  \end{equation}
\eit

\begin{maindifficulty}
  \label{difficile}
The main difficulty dealing with linear functions is as follows. It appears that to the partition $(h\B,h\U)$ of the
set of non null linear functions we need to associate not a unique but two different $\lambda$-parametrized
domains for $(x,h)$. The main consequence is that it does not exist a global hierarchy between $h\B$ linear
functions and  $h\U$ linear functions based on the inclusion of the images of a common $\lambda$-parametrized
domain for $(x,h)$. In other words we have the following property
\[
\exists f \in h\B, \exists  g \in h\U \mbox{ s.t. } f(R_{\lambda}^{\B}(x,h)) \not\subseteq g(R_{\lambda}^{\U}(x,h))
\]
which is verified. Take eg. $f(x_{1},x_{2},h)= -x_{1} + 40h \in h\B$ with
$f(R_{\lambda}^{\B}(x_{1},x_{2},h))= [-40 \lambda, 40 \lambda]$
and $g(x_{1},x_{2},h)= x_{1}-x_{2} + h \in h\U$ with $g(R_{\lambda}^{\U}(x_{1},x_{2},h)=[-3 \lambda, 3 \lambda]$. \\
Let us stress that this case does not appear for $(\max,+)$-linear programming where all the
$(\max,+)$-linear functions are increasing or constant and always positive (ie. $\geq -\infty$, where
$-\infty$ is the zero for the ``addition'' $\max$). So we cannot make a copy-paste of the results of
\cite[Appendix A]{truffet:hal-05556396} dealing with the global hierarchy of the functions denoted
$\preceq_{\overline{fct}}$ in \cite{truffet:hal-05556396}.
\end{maindifficulty}

However, even if there is no global hierarchy the following results will be useful in the sequel. 

\begin{prpstn}[Elements of $h\B$ vs elements of $h\U$]
  \label{propoIneqhBvshU}

  Let \\
  $f_{1}(x,h)=\alpha^{\intercal} x + \beta h$ be a non-null $h$-bounded $(x,h)$-linear function.
  Let $f_{2}(x,h)=\gamma^{\intercal} x + \delta h$ be a non-null $h$-unbounded $(x,h)$-linear function. Then,
  we have the following strict inequality:

\begin{equation}
  \forall x \geq \bzero \; \exists x' \geq \bzero: f_{1}(x,h) \leq \beta h <  f_{2}(x',h). 
  \end{equation}
  \end{prpstn}
\proof Because $f_{2} \in h\U$ there exists $j$ s.t. $\gamma_{j} >0$. Then, define
$x'$ as follows. $\forall j' \neq j$: $x'_{j'}=0$ and
$x'_{j}= \frac{1}{\gamma_{j}}(\beta h - \delta h + \kappa)$ for some
$\kappa$ s.t. $\kappa > 0$ and $\beta h - \delta h + \kappa >0$.

\cqfd \\

However, we remark that if $f(x,h) \in h\B$ then the computation of the
symmetric intervals $f(R^{\B}_{\lambda}(x,h))$ and $f(R^{\U}_{\lambda}(x,h))$ can be sensible.
Indeed, for all variables $x_{j}$ assumed to be in $[0,+\infty]$ then $-x_{j} \geq 0$ and the
only symmetric interval that takes into account this las inequality is $[0,0]$. But
looking for a parametrized research of the maximum one can also
assume that $x_{j} \in [-\lambda,\lambda]$ thus $-x_{j}$ is also in $[-\lambda,\lambda]$.
On the contrary if $f(x,h) \in h\U$ then only the symmetric interval $f(R^{\U}_{\lambda}(x,h))$ is sensible
but not $f(R^{\B}_{\lambda}(x,h))$.

\begin{prpstn}[Conditioned hierarchy between $h\B$ and $h\U$]
  \label{propoIneqhBvshUHierarchy}
  Let \\ $f_{1}(x,h)=\alpha^{\intercal} x + \beta h$ be a non-null $h$-bounded $(x,h)$-linear function.
  Let $f_{2}(x,h)=\gamma^{\intercal} x + \delta h$ be a non-null $h$-unbounded $(x,h)$-linear function. Then,
  we have the following implication to be true:

  \begin{equation}
    f_{1}(R_{\lambda}^{\U}(x,h)) \subseteq f_{2}(R_{\lambda}^{\U}(x,h)) \Rightarrow
    f_{1}(R_{\lambda}^{\B}(x,h)) \subseteq f_{2}(R_{\lambda}^{\U}(x,h)).
    \end{equation}
  
  \end{prpstn}
\proof It is sufficient to note that: $R_{\lambda}^{\B}(x,h) \subseteq R_{\lambda}^{\U}(x,h)$ and
the linear functions are applications.

\cqfd \\


In what follows we will define the partition of the set of variables into two susbsets and the hierarchy
beween these susbsets.

Let us remark that we have the following equivalence:
 \begin{equation}
   \label{equivdomine}
\boldsymbol{H} \mbox{ and } \mathcal{C}(A,b) \cap \real_{+}^{n+1} \neq \{\bzero\} \Leftrightarrow \exists x_{j} >0.
    \end{equation}

 Under $\boldsymbol{H}$ we have the following hierarchy $\preceq_{var}$ between the set of the remaining
 variables of the problem based on the nonnegative reachability domains of the variables from the smallest to
 the greatest (same principle as for $(x,h)$-linear functions involved in the problem):

 \begin{equation}
    \label{eqDomfirst}
    \{x_{j} \in \x: I_{\leq}(x_{j}) \neq \emptyset \mbox{ and } I^{+}_{\geq}(x_{j}) \neq \emptyset \} \preceq_{var} \{x_{j} \in \x: I_{\leq}(x_{j}) \neq \emptyset \mbox{ and } I^{+}_{\geq}(x_{j}) = \emptyset \}.
\end{equation}
Of course we also have: \\
\[
\{x_{j} \in \x: I_{\leq}(x_{j}) \neq \emptyset \mbox{ and } I^{+}_{\geq}(x_{j}) = \emptyset \}
\preceq_{var} \{x_{j} \in \x: I_{\leq}(x_{j}) = \emptyset \mbox{ and } I^{+}_{\geq}(x_{j}) = \emptyset \}.
\]
 
Thus, as for $(\max,+)$-linear programming problem \cite{truffet:hal-05556396} we define a very
similar concept of
dominating variable as follows.
 
  \begin{dfntn}[Dominating variable]
    \label{variableOrdre}
    We say that a variable $x_{j}$ is dominating if $I_{\leq}(x_{j}) \neq \emptyset$ and
    $I^{+}_{\geq}(x_{j}) \neq \emptyset$.
  \end{dfntn}

\subsection{Choosing a variable for substitution after $k$ substitutions}
\label{subchoosesubstit}

In this subsection we will describe the procedure for choosing a remaining variable
vto be substituted. We assume that we have done $k$ substitutions. The polyhedral
cone associated with the problem at step $k$, ie. \textsf{pmrp}($A,b,c,z,x,h$; $(m, n), \textsf{cost}^{[k]}$),
is denoted $\mathcal{C}(\overline{A}^{[k]})$ (see \cref{secconeaveclepmrp}). The set of linear inequalities
$\mathcal{L}^{[k]}$ contains
the linear equalities of the form $\{x_{j}= f\}$, where $f$ is a $(x,h)$-linear $f_{ij}$-function
defined by (\ref{deffij})-(\ref{eqdefrij}). The set $\mathcal{L}^{[k]}$ can also contain equalities
of the form $\{x_{j}=0\}$ because the function $\textsf{checknulvar}$ (\ref{efchecknulvar0})
has been applied at each substitution step from $0$ to $k$. The matrix $\overline{A}^{[k]}$ can also have null row vectors
because the function $\textsf{setrowtozero}$ (\ref{setrowtozero0}) has also been applied at
each substitution step from $0$ to $k$.
The set $\x$ associated with the vector $x$ of the remaining variables (recall that $x_{j}=0$ in $x$ means
  that $x_{j}$ has been substituted and $x_{j} \notin \x$) is partitioned into the following three subsets 
  associated with the cost function and defined by:

  \begin{equation}
    \label{x-partition}
\bar{lll}
\x^{+} &:=& \{x_{k} \in \x: c_{k} > 0\} \\
\x^{0} &:=& \{x_{k} \in \x: c_{k} = 0\} \\
\x^{-} &:=& \{x_{k} \in \x: c_{k} < 0\}.
\ear
\end{equation}
  And the three vectors associated with the prevous subsets are defined by:
\begin{equation}
    \label{x-partitionvecteur}
    x^{+} := (x_{k})_{x_{k} \in \x^{+}}, \; x^{0} := (x_{k})_{x_{k} \in \x^{0}}, \; x^{-} := (x_{k})_{x_{k} \in \x^{-}}.
\end{equation}

\begin{prpstn}[Unbounded problem]
  \label{propUnbounded}
  The \textsf{pmrp} problem has an unbounded maximum if $\boldsymbol{U}$ occurs with: \\

  \noi
  ($\boldsymbol{U}$): $\x^{+} \neq \emptyset$ and $\exists x_{j} \in \x^{+}$ s.t. $\overline{A}_{[|1,m|]j} < \bzero$ \\
where the strict inequality $\overline{A}_{[|1,m|]j} < \bzero$ means: $\forall i \in [|1,m|]$ $a_{i,j} <0$.
  
  \end{prpstn}
\proof If $\overline{A}_{[|1,m|]j} < \bzero$ then $x_{j}=+\infty$ is a possible value for this variable and the
assumption $\textbf{H}$ (see (\ref{hypCAbnontrivial})) is verified.
Thus, if $x_{j} \in \x^{+}$ then
$c_{j} >0$ and the the cost function $c^{\intercal}x + c_{h}h =+\infty$ for $x_{j}=+\infty$.
\cqfd. \\

From now on it is assumed that:

\begin{equation}
  \label{HypprmprBounded}
\boldsymbol{\overline{U}}: \mbox{ the \textsf{pmrp} is bounded}.
  \end{equation}

In the next theorem we present in fact the stopping condition of the \textsf{pmrp} at step
$k$ of the substitution. Recall that we use numerotation convention~\ref{desnotations}
p.\pageref{desnotations} for matrices.

\begin{thm}[Maximality and reachability at $\bzero$ of $h \mapsto c_{h}h$]
  \label{thmReachch.h}
  
  The function $h \mapsto c_{h}h$ with $c_{h} \geq 0$ is the upper bound of the cost function $(x,h) \mapsto \textsf{cost}^{[k]}(x,h)$
which is attained at $x = \bzero$ iff the following conditions hold:

\begin{subequations}
  \begin{equation}
    \label{z=ch0}
\x^{+}  = \emptyset
  \end{equation}
  and 
  \begin{equation}
        \label{z=ch1}
\overline{A}^{[k]}_{[|0,m|]h} \leq \bzero
  \end{equation}
\end{subequations}
\end{thm}
\proof
  Remark that the cost function can be written as: $\textsf{cost}^{[k]}(x,h)=c^{\intercal}_{\x^{+}} x^{+} +
c^{\intercal}_{\x^{0}} x^{0} + c^{\intercal}_{\x^{-}} x^{-} + c_{h} h$. \\
  Proof of ($\Leftarrow$). The set equality: $\x^{+} = \emptyset$ implies that
  $\textsf{cost}^{[k]}(x,h)= c_{\x^{0}} x^{0} + c_{\x^{-}} x^{-} + c_{h} h \leq c_{h}h$ because $c_{\x^{0}} = \bzero$ and
  $c_{\x^{-}} \leq \bzero$. We also have: $\x=\x^{0} \cup \x^{-}$. The $[|0,m|] \times [|1,n+1|]$-system of
  inequalities which characterize the cone $\mathcal{C}(\overline{A}^{[k]})$ is reduced to 
  the following $[|0,m|] \times (\x \cup \{h\})$-system of inequalities:

  \begin{equation}
    \label{eqAxhleq0}
\overline{A}^{[k]}_{[|0,m|]\x} x + \overline{A}^{[k]}_{[|0,m|]h} h \leq \bzero. 
  \end{equation}
  Assuming $\overline{A}^{[k]}_{[|0,m|]h} \leq \bzero$, the point $\left(\bar{c} \bzero \\ h \ear \right)$, with $h\neq 0$
  is solution  of (\ref{eqAxhleq0}). Thus, the maximum function $c_{h} h$ is reached at this point and
  the implication ($\Leftarrow$) is proved. \\

  Proof of ($\Rightarrow$). $h \mapsto c_{h} h$ is an upper bound of the $(x,h)$-linear cost
  function $(x,h) \mapsto \textsf{cost}^{[k]}(x,h)$ means that $\forall x \geq \bzero$: $\textsf{cost}^{[k]}(x,h) \leq c_{h} h$.
v  Thus, $\textsf{cost}^{[k]}$ is a $h$-bounded function. This means that necessarily $c_{\x^{+}}=\bzero$
  or equivalently that $\x^{+}=\emptyset$. The cost function is written as:
  $\textsf{cost}^{[k]}(x,h)= c^{\intercal}_{\x^{0}} x^{0} + c^{\intercal}_{\x^{-}} x^{-} + c_{h} h$. And $c_{h} h$ is reached at $x=\bzero$ if and only if
  $(z, \bzero,h) \in \mathcal{C}(\overline{A}^{[k]})$ which implies in particular that:
  $\overline{A}^{[k]}_{[|0,m|]h} \leq \bzero$.
  The ($\Rightarrow$) is now proved. \\
  Finally, the logical equivalence is now proved.
\cqfd \\

Note that from the ``cost point of view'' the stopping condition of
the \textsf{pmrp} (see Theorem~\ref{thmReachch.h}) is based
on the $h\B$ functions
associated with the $\lambda$-parametrized domain $R^{\B}_{\lambda}(x,h)$.

Let us note that the computation of the symmetric interval of several linear functions $f$'s
depends on the fact that the $f$'s are such that they are upper bounds for the variables
$z$ and $x_{j}$ or lower bounds for the variables $x_{j}$. This remark
leads to the following developments based on
which kinds of inequalities
the linear functions $f$'s are involved. Let us precise the context of inequalities hereafter. 

 Let us define the following set of indexes.

  \begin{equation}
\I^{\leq}:=\{(i,j) \mbox{ s.t. } z \leq f_{z,ij} \mbox{ and } x_{j} \leq f^{\leq}_{ij} \mbox{ both exist}\},
    \end{equation}
  and
    \begin{equation}
\I^{\geq}:=\{(i,j) \mbox{ s.t. } z \leq f_{z,ij} \mbox{ and } x_{j} \geq f^{\geq}_{ij} \mbox{ both exist}\}.
    \end{equation}


We use the conventions
$I_{\leq}(\emptyset)=\emptyset$, $I^{+}_{\geq}(\emptyset)=\emptyset$ and $I_{\geq}(\emptyset)=\emptyset$
and the maximal sets $\I^{\leq}, \I^{\geq}$ are defined as follows.
  
\begin{prpstn}[Computation of the maximal sets $\I^{\leq}, \I^{\geq}$]
  \label{propomaxsetIleqIgeq}
  First compute the set $D$ of dominating variables (see Definition~\ref{variableOrdre})
  that is:
  \[
D=\{x_{j} \in \x: I_{\leq}(x_{j}) \neq \emptyset \mbox{ and } I^{+}_{\geq}(x_{j}) \neq \emptyset \}.
  \]

  \noi
  \textsc{Case 1}. $D \neq \emptyset$.
  Compute:
  
  \begin{subequations}
  \begin{equation}
\I^{\leq}:= I_{\leq}((\x^{+} \cup \x^{0}) \cap D),
    \end{equation}
  where
  \begin{equation}
    I_{\leq}((\x^{+} \cup \x^{0}) \cap D)=\{j: x_{j} \in (\x^{+} \cup \x^{0}) \cap D, i \in [|1,m|]  \mbox{ s.t. }
    I_{\leq}(x_{j},i) \neq \emptyset \}.
    \end{equation}
  
  Compute $\I^{\geq}$ as follows:
    
  \begin{equation}
\I^{\geq}:= I^{+}_{\geq}((\x^{0} \cup \x^{-}) \cap D),
    \end{equation}
  where
  \begin{equation}
    I^{+}_{\geq}((\x^{0} \cup \x^{-}) \cap D)=\{j: x_{j} \in (\x^{0} \cup \x^{-}) \cap D, i \in [|1,m|]  \mbox{ s.t. }
    I^{+}_{\geq}(x_{j},i) \neq \emptyset \}.
    \end{equation}
  \end{subequations}

  \noi
  \textsc{Case 2}. $D=\emptyset$. \\
  Compute:
  
  \begin{subequations}
  \begin{equation}
\I^{\leq}:= I_{\leq}(\x^{+} \cup \x^{0}),
    \end{equation}
  where
  \begin{equation}
I_{\leq}(\x^{+} \cup \x^{0}) =\{j: x_{j} \in \x^{+} \cup \x^{0}, i \in [|1,m|]  \mbox{ s.t. } I_{\leq}(x_{j},i) \neq \emptyset \}.
    \end{equation}

 Compute $\I^{\geq}$ as follows:

  \begin{equation}
    \I^{\geq}:= \left\{ \bar{ll} I^{+}_{\geq}(\x^{0} \cup \x^{-}) & \mbox{ if $I^{+}_{\geq}(\x^{0} \cup \x^{-})\neq \emptyset$} \\
    I_{\geq}(\x^{0} \cup \x^{-}) & \mbox{ otherwise}.

    \ear \right.
    \end{equation}
  where
  \begin{equation}
I^{+}_{\geq}(\x^{0} \cup \x^{-}) =\{j: x_{j} \in \x^{0} \cup \x^{-}, i \in [|1,m|]  \mbox{ s.t. } I^{+}_{\geq}(x_{j},i) \neq \emptyset \}
  \end{equation}
  and
    \begin{equation}
I_{\geq}(\x^{0} \cup \x^{-}) =\{j: x_{j} \in \x^{0} \cup \x^{-}, i \in [|1,m|] \mbox{ s.t. } I_{\geq}(x_{j},i) \neq \emptyset \}
  \end{equation}
\end{subequations}

  \end{prpstn}
\proof Due to the hierarchy (see (\ref{eqDomfirst})) between variables of the problem it is clear
that we first take into account the dominating variables, ie set $D$ in the Proposition (see
also Definition~\ref{variableOrdre}). The inequalities $I_{\leq}(x_{j},i)$ dominate (w.r.t. to the
\textsf{positive property} defined by (\ref{fijpositivetoutletemps},
Definition~\ref{defpositivproperty}) the
inequalities $I_{\geq}(x_{j},i)$ because all the functions $f^{\leq}_{ij}$ have the
\textsf{positive property} as a result of Proposition~\ref{propfijpositive}. Moreover,
the upper bounds if exist
are always better because they define a maximum domain associated with the cone
$\mathcal{C}(\overline{A})$ of the \textsf{pmrp} in the sense of the set inclusion (see result
of Proposition~\ref{propUpperBdfirst}).

\cqfd. \\

In the next Proposition we discuss properties of the series $\textsf{cost}^{[k]}, 0 \leq k \leq n$, for
a given $\textsf{cost}^{[0]}$ which is the initial $x$-linear cost function to positively maximize.

\begin{prpstn}
  \label{propoPptesdesCosts}
  We have the following noticeable properties.

  \bit
\item (1). If $\textsf{cost}^{[0]}$ is the null function then $\forall k \geq 1$ $\textsf{cost}^{[k]}$ is null.

  \item (2). $\forall k$ s.t. $\textsf{cost}^{[k]} \in h\B$ and $\textsf{cost}^{[k+1]} \in h\B$:

    \begin{equation}
      \label{eqcoutshbInclus}
    \textsf{cost}^{[k+1]}(R^{\B}_{\lambda}(x,h)) \subseteq \textsf{cost}^{[k]}(R^{\B}_{\lambda}(x,h))
  \end{equation}

    \item (3). If $\exists k$ s.t. $\textsf{cost}^{[k]} \in h\B$ then $\forall k' \geq k$: $\textsf{cost}^{[k']}$
  must be an element of $h\B$.
  
  \eit
  \end{prpstn}
\proof \\
Proof of (1). By construction of the costs where a variable $x_{j}$ is replaced by its $f_{ij}$-function in the
expression of the previous cost. If the previous cost is the null linear combination of the $x_{j}$'s necessarily
the next cost is also null. \\

\noi
Proof of (2). We have $\textsf{cost}^{[k]}(x,h)= \alpha_{\x}^{\intercal} x + c^{[k]}_{h} h$ and
$\textsf{cost}^{[k+1]}(x',h)= \beta_{\x'}^{\intercal}x' + c^{[k+1]}_{h} h$ with $\alpha_{\x}, \beta_{\x'} \leq \bzero$
and $\x' \subset \x$ with $\textsf{n}(\x')=\textsf{n}(\x)-1$, where $\x$ (resp. $\x'$) denotes the set
of the remaining variables at
step $k$ (resp. $k+1$) of the substitution process. At step $k$ we already have $z \leq c^{[k]}_{h} h$ and
because $x \mapsto \alpha_{\x}^{\intercal} x$ is non-increasing necessarily $c^{[k]}_{h} h$ is the maximum
reachable cost. So at step $k+1$ the maximum reachable cost, ie. $c^{[k+1]}_{h} h$, cannot be
strictly greater than the previous maximum reachable cost and we have: $c^{[k+1]}_{h} h \leq c^{[k]}_{h} h$.
To conclude we just have to note that the last inequality is equivalent to the set
inclusion (\ref{eqcoutshbInclus}).  \\

\noi
Proof of (3). This ``constraint'' is a consequence of (2), the stopping condition of
the substitution procedure (see Theorem~\ref{thmReachch.h}) which is associated
with a cost function which is an element of the set $h\B$ and the study
of $h\B$ functions vs $h\U$ functions (see Proposition~\ref{propoIneqhBvshU}). \\

\cqfd \\

From now on till the end the upper bound on $z$ which is obtained by replacing the variable
$x_{j}$ by either the $f_{ij}$-function $f^{\leq}_{ij}(x,h)$ or $f^{\geq}_{ij}(x,h)$ will be denoted
$f_{z,ij} (x,h)$ or for short $f_{z,ij}$. And we have the most important results of the paper presented in the next
two theorems. We characterize the set of substituable variables from a given set of inequalities: upper bounds
on the priority variable $z$ and upper or lower bounds on the remaining variables $x_{j}$ of the
problem. All the results are presented in the $(z,x,h)$-system of variables. Indeed, the
$(z,-x,-h)$-system of variables is only needed for the proof of Theorem~\ref{propVarsubstitposs2}. 

We denote $\textsf{cost}^{[k]}(x,h)$ the current cost function which is the upper bound of the variable $z$
obtained after $k$ substitutions of variables $x_{j'}$'s of the LPP. Then, $f_{z,ij}$ could be considered
as the possible new cost function obtained by the substitution of the variable $x_{j}$ by the function
$f_{ij}$-function $f^{\leq}_{ij}(x,h)$ or $f^{\geq}_{ij}(x,h)$ in $(x,h) \mapsto \textsf{cost}^{[k]}(x,h)$. 

We define the following sets of linear functions:

\begin{subequations}
\begin{equation}
  \label{eqdefCostkIleq}
\textsf{Cost}(k,\I^{\leq}):=\{\textsf{cost}^{[k]}\} \cup_{(i,j) \in \I^{\leq}} \{f_{z,ij}\},
\end{equation}
and
\begin{equation}
  \label{eqdefCostkIgeq}
\textsf{Cost}(k,\I^{\geq}):=\{\textsf{cost}^{[k]}\} \cup_{(i,j) \in \I^{\geq}} \{f_{z,ij}\},
\end{equation}

\end{subequations}

Let $\I$ denote either $\I=\I^{\leq}$ or $\I=\I^{\geq}$. \\

\noi
Let us define the following function called
$\textsf{Nearest-to-zero-criterion-for}(\textsf{Cost}(k,\I))$ as follows.

\noi
$\textsf{Nearest-to-zero-criterion-for}(\textsf{Cost}(k,\I))$:
\label{nearestfct}

\bit

\item Output: the function returns a symmetric interval $\upsilon$ element of  \\
  $\{[0,0], [-\alpha \lambda, \alpha \lambda],
  [-\infty,\infty]\}$ for some $\alpha >0$.

  \eit

  \noi
  {\bf Begin} \\

\noi
If the the null linear function $0(\cdot) \in \textsf{Cost}(k,\I)$ then the
function returns $\upsilon=[0,0]$ else let us define the following cases that may occur: \\

\noi
case (a). $\textsf{cost}^{[k]} \in h\U$ and $\forall (i,j) \in \I$: $f_{z,ij} \in h\B$ \\
case (b). $\textsf{cost}^{[k]} \in h\U$ and $\exists (i,j) \in \I$: $f_{z,ij} \in h\U$ \\ 
case (c). $\textsf{cost}^{[k]} \in h\B$.  \\

Let us define the set denoted $\mathcal{T}(k, \I)$ depending on the
previous cases (a), (b), (c) as follows:

\begin{equation}
  \label{defTkI}
  \mathcal{T}(k, \I):=\left\{\bar{lr} \I & \mbox{(a)} \\
  \I &  \mbox{(b)} \\
  \{(i,j) \in \I: \; f_{z,ij} \in h\B
      \mbox{ and } f_{z,ij}(R^{\B}_{\lambda}(x,h)) \subseteq
      \textsf{cost}^{[k]}(R^{\B}_{\lambda}(x,h))\} & \mbox{(c)}
  \ear \right. .
\end{equation}

If $\mathcal{T}(k, \I) \neq \emptyset$ then the function returns:

\noi
case (a) or (c): \\

\begin{subequations}
  \begin{equation}
    \label{eqtau}
\upsilon= \textsf{Min}_{\{i:(i,j) \in \mathcal{T}(k, \I)\}} \upsilon_{i},
    \end{equation}
    with $\upsilon_{i}$ defined by:
    \begin{equation}
      \label{eqtaui}
\upsilon_{i}= \textsf{Max}_{\{j:(i,j) \in \mathcal{T}(k, \I)\}} f_{z,ij}(R^{\B}_{\lambda}(x,h))
    \end{equation}
\end{subequations}

    \noi
    case (b): \\
    Compute interval $\rho$ defined by:
    
  \begin{subequations}
  \begin{equation}
    \label{eqrho}
\rho= \textsf{Min}_{\{i:(i,j) \in \mathcal{T}(k, \I)\}} \rho_{i},
    \end{equation}
    with $\rho_{i}$ defined by:
    \begin{equation}
      \label{eqtrhoi}
\rho_{i}= \textsf{Max}_{\{j:(i,j) \in \mathcal{T}(k, \I)\}} f_{z,ij}(R^{\U}_{\lambda}(x,h))
    \end{equation}  
\end{subequations}    


  We have the following subcases: \\
  
  \noi
  case (b. $\U$): the conditioned hierarchy does not apply, ie. $\textsf{argMin}(\rho) \cap h\B = \emptyset$ \\
  The function returns $\upsilon=\rho$. \\
  
  \noi
   case (b. $\B$): the conditioned hierarchy applies, ie. $\textsf{argMin}(\rho) \cap h\B \neq \emptyset$ \\
   The function returns $\upsilon$ defined by:

   \begin{subequations}
  \begin{equation}
    \label{eqtaucondHierarchy}
\upsilon= \textsf{Min}_{\{i:(i,j) \in \textsf{argMin}(\rho) \cap h\B\}} \upsilon_{i},
    \end{equation}
    with $\tau_{i}$ defined by:
    \begin{equation}
      \label{eqtauicondHierarchy}
\upsilon_{i}= \textsf{Max}_{\{j:(i,j) \in \textsf{argMin}(\rho) \cap h\B\}} f_{z,ij}(R^{\B}_{\lambda}(x,h))
    \end{equation}  
   \end{subequations}
  
    Else the function returns $\upsilon = [-\infty,\infty]$. \\

    \noi
    {\bf End}. \\

    To prove important properties of the function $\textsf{Nearest-to-zero-criterion-for}$ we
    need to define the following logical decomposition of the sets $\mathcal{T}(k, \I^{\leq})$
    and $\mathcal{T}(k, \I^{\geq})$ as follows.
     We adopt the notation conventions: $\textsf{AND}_{\{i \in \{i_{1}, \ldots, i_{l}\}\}} P(i):= P(i_{1}) \textsf{and} \cdots \textsf{and} P(i_{l})$ and
$\textsf{OR}_{\{j \in \{j_{1}, \ldots, j_{k}\}\}} Q(j):= Q(j_{1}) \textsf{or} \cdots \textsf{or} Q(j_{k})$ with
$P(\cdot)$ and $Q(\cdot)$ are logical propositions (here inequalities which have to be verified). 

\bit 
\item The logical decomposition of the set $\mathcal{T}(k, \I^{\leq})$ according to the
  variable $z$ denoted $\mathcal{T}^{z}_{log}(k, \I^{\leq})$ is defined by:

  \begin{subequations}
\begin{equation}
  \label{eqdecTkIlogleq0}
  \mathcal{T}^{z}_{log}(k, \I^{\leq}):= \textsf{AND}_{\{i:(i,j) \in \mathcal{T}(k, \I^{\leq})\}} \mathcal{T}^{z}_{log}(k, \I^{\leq})(i),
\end{equation}
with
\begin{equation}
    \label{eqdecTkIlogleq1}
    \mathcal{T}^{z}_{log}(k, \I^{\leq})(i):= \textsf{OR}_{\{j:(i,j) \in \mathcal{T}(k, \I^{\leq})\}} (z \leq f_{z,ij})).
      \end{equation}

  \end{subequations}

  This decomposition is relevant because the inequality $z \leq f_{z,ij}$ exists as soon as
  $I_{\leq}(x_{j},i) \neq \emptyset$ and the cost associated with the variable $x_{j}$ is such that $c_{j} \geq 0$.
  
  \item The logical decomposition of the set $\mathcal{T}(k, \I^{\geq})$ according to the
  variable $z$ denoted $\mathcal{T}^{z}_{log}(k, \I^{\geq})$
  follows the logical decomposition $\I_{log}^{\geq}$ and is defined by:

  \begin{subequations}
\begin{equation}
  \label{eqdecTkIloggeq0}
  \mathcal{T}^{z}_{log}(k, \I^{\geq}):= \textsf{AND}_{\{i:(i,j) \in \mathcal{T}(k, \I^{\geq})\}} \mathcal{T}^{z}_{log}(k, \I^{\geq})(i),
\end{equation}
with
\begin{equation}
    \label{eqdecTkIloggeq1}
    \mathcal{T}^{z}_{log}(k, \I^{\geq})(i):= \textsf{OR}_{\{j:(i,j) \in \mathcal{T}(k, \I^{\geq})\}} (z \leq f_{z,ij})).
      \end{equation}

  \end{subequations}

  This decomposition is relevant because the inequality $z \leq f_{z,ij}$ exists as soon as
  $I_{\geq}(x_{j},i) \neq \emptyset$ and the cost associated with the variable $x_{j}$ is such that $c_{j} \leq 0$.

\eit

Where the notation $(z \leq f_{z,ij})$ means $\exists (x,h) \in R^{\T}_{\lambda}(x,h): (z \leq f_{z,ij}(x,h))$.

The function $\textsf{Nearest-to-zero-criterion-for}$ (see p. \pageref{nearestfct}) satisfies the
following assertions presented in the next proposition.

    \begin{prpstn}
      \label{propntzfctppte}
      Assume that $\mathcal{T}(k, \I)$ (\ref{defTkI}) is $\neq \emptyset$.
      
Let us define the following assertion which deals with the cost $z$ and a symmetric interval $\sigma$:
   
   \begin{equation}
     \label{eqZTtauleq}
     \bar{l}
     \mathrm{Z}(\sigma, \I): \forall z \in \sigma \;  \mbox{$\mathcal{T}^{z}_{log}(k, \I)$ is true}.
     \ear
   \end{equation}

\noi   
(A). $\textsf{Max}_{\sigma \in \mathcal{I}_{0}} (\mbox{$\mathrm{Z}(\sigma, \I)$ is true})$ is reached
at $\sigma=\tau$ with: \\
$\tau=\textsf{Nearest-to-zero-criterion-for}(\textsf{Cost}(k,\I))$. \\

   \noi
(B). There exists $k$ s.t. $\textsf{cost}^{[k]} \in h\U$ and $\textsf{cost}^{[k+1]} \in h\B$ if
  the following condition holds. \\
\begin{subequations}
  \begin{equation}
\forall (i,j)  \in \textsf{argMin}(\tau) \; f_{z,ij} \in h\B
  \end{equation}
where:
  \begin{equation}
    \tau=\textsf{Nearest-to-zero-criterion-for}(\textsf{Cost}(k,\I)).
    \end{equation}

\end{subequations}

\noi
(C) The function $\textsf{Nearest-to-zero-criterion-for}$ satisfies the inclusion property (2)
of Proposition~\ref{propoPptesdesCosts}.

\end{prpstn}
    \proof Let us prove (A). The reachability result (A) comes from an extension explained after the
    proof of the results of the following lemma.

    \begin{lem}
      \label{lemmInegalitesleq}
      Let us consider two $(x,h)$-linear functions $f_{1}$ and $f_{2}$, a variable $y$ and the
      two inequalities $I_{\leq}(y,1):=(y \leq f_{1})$ and $I_{\leq}(y,2):=(y \leq f_{2})$. Let
      us also define the symmetric intervals $\sigma_{1}:=f_{1}(R^{\T}_{\lambda}(x,h))=[-\lambda_{1}, \lambda_{1}]$ and
      $\sigma_{2}:=f_{2}(R^{\T}_{\lambda}(x,h))=[-\lambda_{2}, \lambda_{2}]$, $\lambda_{1}, \lambda_{2}$ linear functions
      of $\lambda$ such that $\lambda_{1}, \lambda_{2} >0$. We have the
      following results.

      \bit

    \item (1). The maximum symmetric interval $\textsf{Max}_{\sigma \in \mathcal{I}_{0}}
      (\forall y \in \sigma \; \exists (x,h) \in R^{\T}_{\lambda}(x,h) \; (y \leq f_{i}(x,h)) \mbox{ is true})$ is equal to
      $\sigma_{i}$, $i=1,2$.

          \item (2). The maximum symmetric interval $\textsf{Max}_{\sigma \in \mathcal{I}_{0}}
            (\forall y \in \sigma \; \exists (x,h) \in R^{\T}_{\lambda}(x,h)
            \; (y \leq f_{1}(x,h)) \mbox{ and } (y \leq f_{2}(x,h)) \mbox{ is true})$ is equal to
            $\textsf{Min}(\sigma_{1}, \sigma_{2})$.

             \item (3). The maximum symmetric interval $\textsf{Max}_{\sigma \in \mathcal{I}_{0}}
            (\forall y \in \sigma \; \exists (x,h) \in R^{\T}_{\lambda}(x,h) \;
            (y \leq f_{1}(x,h)) \mbox{ or } (y \leq f_{2}(x,h)) \mbox{ is true})$ is equal to
      $\textsf{Max}(\sigma_{1}, \sigma_{2})=\sigma_{1} \cup \sigma_{2}$.

\eit
      
      \end{lem}
    \proof \\
    Proof of (1). It is sufficient to remark that if $y > \lambda_{i}$ then by definition the image of a
    set by an application
    we have: $\forall (x,h) \in R^{\T}_{\lambda}(x,h) \; y \nleq f_{i}(x,h)$. \\
    Proof of (2). We just have to note: $y \in \sigma_{1}$ and $y \in \sigma_{2}$ is equivalent to
    $y \in \sigma_{1} \cap \sigma_{2}$. And $\sigma_{1} \cap \sigma_{2}=\textsf{Min}(\sigma_{1}, \sigma_{2})$ by definition
    of $\textsf{Min}$. \\
    Proof of (3).  Looking for positive maximum on positive domain for the variables we need to focus on the
    positive part of the intervals and we must have in particular: $y \leq \lambda_{1}$ or $y \leq \lambda_{2}$.
    From the result
    of Proposition~\ref{propUpperBdfirst} and its proof upper bounds are better that lower bounds and of course the greatest upper bound is the best. Here the
    best upper bound is $\max(\lambda_{1}, \lambda_{2})$. To conclude we just have to note that this upper bound is
    also the upper bound of the symmetric interval $\textsf{Max}(\sigma_{1}, \sigma_{2})$. \\

    \cqfd \\
    The extension of this lemma is as follows. The expression of $\tau_{i}$ is an extension of
    (3). The computation of $\tau$ is an extension of (2) where the $\sigma_{i}$'s are replaced by the
    $\tau_{i}$'s noticing that the number of intervals $\tau_{i}$ can be
    different from the number of sets $\sigma_{i}$ (but $\textsf{Min}$ and $\textsf{Max}$ are associative).
    Finaly, the variable $y$ is replaced by the cost variable $z$. \\
    
\noi
The assertion (B) is clearly a consequence of (A) and the fact that the
next cost computed from the current $\textsf{cost}^{[k]}$ is necessarily a function
$f_{z,ij}$.  \\

The assertion (C) is true by construction of the set $\mathcal{T}(k,\I)$ (see (\ref{defTkI}), case (c)). 

\cqfd \\

Before stating the next Theorem we need the following preliminary Lemma.

    \begin{lem}
      \label{lemmInegalitesgeq}
      Let us consider two $(x,h)$-linear functions $f_{1}$ and $f_{2}$, a variable $y$ and the
      two inequalities $I_{\geq}(y,1):=(y \geq f_{1})$ and $I_{\geq}(y,2):=(y \geq f_{2})$. Let
      us also define the symmetric intervals $\sigma_{1}:=f_{1}(R^{\T}_{\lambda}(x,h))=[-\lambda_{1}, \lambda_{1}]$ and
      $\sigma_{2}:=f_{2}(R^{\T}_{\lambda}(x,h))=[-\lambda_{2}, \lambda_{2}]$, $\lambda_{1}, \lambda_{2}$ linear functions
      of $\lambda$ such that $\lambda_{1}, \lambda_{2} >0$. We have the
      following results.

      \bit

    \item (1'). The maximum symmetric interval $\textsf{Max}_{\sigma \in \mathcal{I}_{0}}
      (\forall y \in \sigma \; \exists (x,h) \in R^{\T}_{\lambda}(x,h) \; (y \geq f_{i}(x,h)) \mbox{ is true})$ is equal to
      $\sigma_{i}$, $i=1,2$.

          \item (2'). The maximum symmetric interval $\textsf{Max}_{\sigma \in \mathcal{I}_{0}}
            (\forall y \in \sigma \; \exists (x,h) \in R^{\T}_{\lambda}(x,h) \;
            (y \geq f_{1}(x,h)) \mbox{ and } (y \geq f_{2}(x,h)) \mbox{ is true})$ is equal to
            $\textsf{Max}(\sigma_{1}, \sigma_{2})$.

             \item (3'). The maximum symmetric interval $\textsf{Max}_{\sigma \in \mathcal{I}_{0}}
            (\forall y \in \sigma \; \exists (x,h) \in R^{\T}_{\lambda}(x,h) \;
            (y \geq f_{1}(x,h)) \mbox{ or } (y \geq f_{2}(x,h)) \mbox{ is true})$ is equal to
      $\textsf{Min}(\sigma_{1}, \sigma_{2})$.

\eit
      
      \end{lem}
    \proof \\
    Proof of (1'). It is sufficient to remark that if $y < -\lambda_{i}$ then by definition the image of a set by an application
    we have: $\forall (x,h) \in R^{\T}_{\lambda}(x,h) \; y \ngeq f_{i}(x,h)$. \\
    Proof of (2'). Looking for positive maximum on positive domain for the variables we need to focus on the
    positive part of the intervals and we must have in particular: $y \geq \lambda_{1}$ and $y \geq \lambda_{2}$. From
    the proof of Proposition~\ref{propUpperBdfirst} and its proof for lower bounds we have to take
    the smallest lower bound which is $\max(\lambda_{1}, \lambda_{2})$. To conclude we remark that this bound
    is the upper bound of the symmetric interval $\textsf{Max}(\sigma_{1}, \sigma_{2})$. 
    
    Proof of (3'). Looking for positive maximum on positive domain for the variables we need to focus on the
    positive part of the intervals and we must have in particular:  $y \geq \lambda_{1}$ or $y \geq \lambda_{2}$. From
    the result of Proposition~\ref{propUpperBdfirst} and its proof for lower bounds the smallest lower bound is
    the best one. Here the
    best lower bound is $\min(\lambda_{1}, \lambda_{2})$. To conclude we just have to note that this lower bound is
    also the positive upper bound of the symmetric interval $\textsf{Min}(\sigma_{1}, \sigma_{2})$. \\

    \cqfd \\

\begin{thm}[The set of the substituable variables after $k$ substitutions]
  \label{thmNTZ}
  The set of substituable variables is obtained by the following
  algorithm.

  \noi
  Compute:
  \begin{equation}
    \label{eqthmtau}
    \tau=\textsf{Nearest-to-zero-criterion-for}(\textsf{Cost}(k,\I^{\leq})),
  \end{equation}
  where the set of functions $\textsf{Cost}(k,\I^{\leq})$ is defined by (\ref{eqdefCostkIleq}). \\
  If $\tau \neq [-\infty,\infty]$ then the following cases may occur: \\
  case (a'). $\forall (i,j) \in \textsf{argMin}(\tau)$: $f^{\leq}_{ij} \in h\B$ \\
  case (b'). $\exists (i,j) \in \textsf{argMin}(\tau)$: $f^{\leq}_{ij} \in h\U$ \\
  Define the symmetric interval denoted  $\tau'$ depending on the
  previous cases (a'), (b') as follows. \\

  case (a'): \\
We define $\tau'$ as follows:

  \begin{subequations}
    \begin{equation}
         \label{eqthmtauprime}
\tau':= \textsf{Min}_{\{i: (i,j) \in \textsf{argMin}(\tau)\}} \tau'_{i},
    \end{equation}
    with
  \begin{equation}
    \label{eqthmtauprimei}
    \tau'_{i}:=\textsf{Max}_{\{j:(i,j) \in \textsf{argMin}(\tau)\}}f^{\leq}_{ij}(R^{\B}_{\lambda}(x,h)),
      \end{equation}
  \end{subequations}

  \noi
  case (b'): \\
    Compute interval $\rho'$ defined by:
    
  \begin{subequations}
  \begin{equation}
    \label{eqrhoprime}
\rho'= \textsf{Min}_{\{i:(i,j) \in \textsf{argMin}(\tau)\}} \rho'_{i},
    \end{equation}
    with $\rho'_{i}$ defined by:
    \begin{equation}
      \label{eqtrhoprimei}
\rho'_{i}= \textsf{Max}_{\{j:(i,j) \in \textsf{argMin}(\tau)\}} f^{\leq}_{ij}(R^{\U}_{\lambda}(x,h))
    \end{equation}  
\end{subequations}    
 We have the following subcases: \\
  
\noi
  case (b'. $\U$): the conditioned hierarchy does not apply, ie. $\textsf{argMin}(\rho) \cap h\B = \emptyset$ \\
  Then,  $\tau'=\rho'$. \\
  
 \noi
   case (b. $\B$): the conditioned hierarchy applies, ie. $\textsf{argMin}(\rho) \cap h\B \neq \emptyset$ \\
   Then we compute $\tau'$ defined by:

   \begin{subequations}
  \begin{equation}
    \label{eqtauprimecondHierarchy}
\tau'= \textsf{Min}_{\{i:(i,j) \in \textsf{argMin}(\rho') \cap h\B\}} \tau'_{i},
    \end{equation}
    with $\tau'_{i}$ defined by:
    \begin{equation}
      \label{eqtauprimeicondHierarchy}
\tau'_{i}= \textsf{Max}_{\{j:(i,j) \in \textsf{argMin}(\rho') \cap h\B\}} f^{\leq}_{ij}(R^{\B}_{\lambda}(x,h))
    \end{equation}  
   \end{subequations}

The set of the next substituable variables after $k$ substitution is

  \begin{equation}
    \label{eqsetsubstitvarleq}
S_{\leq}^{[k+1]}:= \textsf{argMin}(\tau')
\end{equation}
 
Else \\
Compute:
\begin{equation}
  \label{eqthmtheta}
  \theta=\textsf{Nearest-to-zero-criterion-for}(\textsf{Cost}(k,\I^{\geq})),
\end{equation}
where the set of functions $\textsf{Cost}(k,\I^{\geq})$ is defined by (\ref{eqdefCostkIgeq}). \\
 If $\theta \neq [-\infty,\infty]$ then the following cases may occur: \\
 case (a''). $I_{\geq}^{+}(\x^{0} \cup \x^{-}) \neq \emptyset$ and
 $\I^{\geq} \subseteq I_{\geq}^{+}(\x^{0} \cup \x^{-})$ \\
 case (b''). $I_{\geq}^{+}(\x^{0} \cup \x^{-}) =\emptyset$ and $I_{\geq}(\x^{0} \cup \x^{-}) \neq \emptyset$ and
 $\I^{\geq} \subseteq I_{\geq}(\x^{0} \cup \x^{-})$ and
 $\forall (i,j) \in \textsf{argMin}(\theta)$: $f^{\geq}_{ij} \in h\B$ \\
  case (c''). $I_{\geq}^{+}(\x^{0} \cup \x^{-}) =\emptyset$ and $I_{\geq}(\x^{0} \cup \x^{-}) \neq \emptyset$ and
  $\I^{\geq} \subseteq I_{\geq}(\x^{0} \cup \x^{-})$ and
  $\exists (i,j) \in \textsf{argMin}(\theta)$: $f^{\geq}_{ij} \in h\U$ \\
  Define the symmetric interval denoted  $\theta'$ depending on the
  previous cases (a''), (b''), (c'') as follows. \\

  \noi
  case (a'') or (b''): \\
  Let us compute $\theta'$ defined by:
  
  \begin{subequations}
    \begin{equation}
          \label{eqthmthetaprime}
\theta':= \textsf{Max}_{\{i: (i,j) \in \textsf{argMin}(\theta)\}} \theta'_{i},
      \end{equation}
  with
  \begin{equation}
     \label{eqthmthetaprimei}
     \theta'_{i}:=  \textsf{Min}_{\{j: (i,j) \in \textsf{argMin}(\theta)\}}f^{\geq}_{ij}(R^{\B}_{\lambda}(x,h)),
       \end{equation}
  \end{subequations}

  \noi
  case (c''): \\
  compute the interval $\rho''$ defined by:

    \begin{subequations}
      \begin{equation}
             \label{eqthmrhosecond}
\rho'':= \textsf{Max}_{\{i: (i,j) \in \textsf{argMin}(\theta)\}} \rho''_{i},
      \end{equation}
  with
  \begin{equation}
     \label{eqthmrhosecondi}
     \rho''_{i}:=  \textsf{Min}_{\{j: (i,j) \in \textsf{argMin}(\theta)\}}f^{\geq}_{ij}(R^{\U}_{\lambda}(x,h)),
       \end{equation}
  \end{subequations}

 We have the following subcases: \\
  
\noi
  case (c''. $\U$): the conditioned hierarchy does not apply, ie. $\textsf{argMax}(\rho'') \cap h\B = \emptyset$ \\
  Then,  $\theta'=\rho''$. \\

  \noi
   case (c''. $\B$): the conditioned hierarchy applies, ie. $\textsf{argMax}(\rho'') \cap h\B \neq \emptyset$ \\
   Then, we compute $\theta'$ defined by:

  \begin{subequations}
  \begin{equation}
    \label{eqthetaprimecondHierarchy}
\theta'= \textsf{Max}_{\{i:(i,j) \in \textsf{argMin}(\rho'') \cap h\B\}} \theta'_{i},
    \end{equation}
    with $\theta'_{i}$ defined by:
    \begin{equation}
      \label{eqthetaprimeicondHierarchy}
\theta'_{i}= \textsf{Min}_{\{j:(i,j) \in \textsf{argMin}(\rho'') \cap h\B\}} f^{\geq}_{ij}(R^{\B}_{\lambda}(x,h))
    \end{equation}  
   \end{subequations}

The set of the next substituable variables after $k$ substitution is

  \begin{equation}
       \label{eqsetsubstitvargeq}
       S_{\geq}^{[k+1]}:= \textsf{argMax}(\theta')
         \end{equation}

\noi
Else \textsf{pmrp} fails.
\end{thm}

\proof The computation of $\tau$ associated with the set $\I^{\leq}$ (computed
in Proposition~\ref{propomaxsetIleqIgeq}) has to be done first.
Recall that as the result of Proposition~\ref{propUpperBdfirst}) the upper bounds if exist
are always better because they define a maximum domain associated with the cone
$\mathcal{C}(\overline{A})$ of the \textsf{pmrp} in the sense of the set inclusion.

The expression of $\tau$ (resp. $\theta$) defined by (\ref{eqthmtau}) (resp. by (\ref{eqthmtheta}))
is due to the property (A), Proposition~\ref{propntzfctppte} of the
function $\textsf{Nearest-to-zero-criterion-for}$ (see p. \pageref{nearestfct}).

To prove the result for $\tau'$ let us define the logical decomposition
of $\textsf{argMin}(\tau)$ as follows.

  \begin{subequations}
    \begin{equation}
         \label{eqthmargmintaulog}
\textsf{argMin}_{log}(\tau):= \textsf{AND}_{\{i: (i,j) \in \textsf{argMin}(\tau)\}} \textsf{argMin}_{log}(\tau)(i),
    \end{equation}
    with
  \begin{equation}
    \label{eqthmargmintauilog}
    \textsf{argMin}_{log}(\tau)(i):=\textsf{OR}_{j\{:(i,j) \in \textsf{argMin}(\tau)\}} (x_{j} \leq f^{\leq}_{ij}).
      \end{equation}
  \end{subequations}

If $\textsf{n}(\textsf{argMin}(\tau)) \geq 2$ then $\forall i \in \{i: (i,j) \in \textsf{argMin}(\tau)\}$
all the variables $x_{j}$, $j \in \{j: (i,j) \in \textsf{argMin}(\tau) \}$
can be exchanged. It means that any inequality $(x_{j} \leq f^{\leq}_{ij}))$ can be replaced by
a generic inequality $(y \leq f^{\leq}_{ij}))$ where $y$ is a generic variable which stands
for any $x_{j}$, $j \in \{j: (i,j) \in \textsf{argMin}(\tau)\}$. In other words $\textsf{argMin}_{log}(\tau)(i)$ can be
expressed as:

\[
\textsf{argMin}_{log}(\tau)(i) = \textsf{OR}_{j\{:(i,j) \in \textsf{argMin}(\tau)\}} (y \leq f^{\leq}_{ij}).
\]

  Now, clearly, by the generalization of the result ((3), Lemma~\ref{lemmInegalitesleq}) we have the following
  equivalence to be true.

  \begin{equation}
    \mbox{$\textsf{argMin}_{log}(\tau)(i)$ is true for $\tau'_{i} \in \mathcal{I}_{0}$} \Leftrightarrow \mbox{ $\tau'_{i}$ is defined by
      (\ref{eqthmtauprimei})}.
    \end{equation}
  Now, we remark that $\textsf{argMin}(\tau'_{i})$ only contains indices $(i,j)$ which correspond to
  inequalities of the form $y \leq f^{\leq}_{ij}$ which have the same feasability domain $\tau'_{i}$ by definition
  of $\textsf{argMin}$. Thus, from the interval point of view this set of inequalities can be considered as a singleton
  $(y \leq f^{\leq}_{ij})$. Thus, we have to compute the maximum interval such that $\textsf{argMin}_{log}(\tau)$ is true which
  is equivalent to compute the maximum interval such that: $\forall i \in \{i: (i,j) \in \textsf{argMin}(\tau)\}
  (y \leq f^{\leq}_{ij}) \mbox{ is true}$. And by the generalization of the result ((2), Lemma~\ref{lemmInegalitesleq})
   we have the following equivalence to be true.

    \begin{equation}
    \mbox{$\textsf{argMin}_{log}(\tau)$ is true for $\tau' \in \mathcal{I}_{0}$} \Leftrightarrow \mbox{ $\tau'$ is defined by
      (\ref{eqthmtauprime})}.
    \end{equation}
  
  The result for $\tau'$ is now proved in the case (a'). The result for $\tau'$ in the case (b') is proved
  using the same arguments. The only difference deals with the set of indices we consider notably when
  conditioned hierarchy (see Proposition~\ref{propoIneqhBvshUHierarchy}) applies (cf. case (b'. $\B$). \\

  \noi
  Let us prove results for the interval $\theta'$. \\
  
Let us define the logical decomposition
of $\textsf{argMin}(\theta)$ as follows.

  \begin{subequations}
    \begin{equation}
         \label{eqthmargminthetalog}
\textsf{argMin}_{log}(\theta):= \textsf{AND}_{\{i: (i,j) \in \textsf{argMin}(\theta)\}} \textsf{argMin}_{log}(\theta)(i),
    \end{equation}
    with
  \begin{equation}
    \label{eqthmargminthetailog}
    \textsf{argMin}_{log}(\theta)(i):=\textsf{OR}_{j\{:(i,j) \in \textsf{argMin}(\theta)\}} (x_{j} \geq f^{\geq}_{ij}).
      \end{equation}
  \end{subequations}

If $\textsf{n}(\textsf{argMin}(\theta)) \geq 2$ then $\forall i \in \{i: (i,j) \in \textsf{argMin}(\theta)\}$
all the variables $x_{j}$, $j \in \{j: (i,j) \in \textsf{argMin}(\theta) \}$
can be exchanged. It means that any inequality $(x_{j} \geq f^{\geq}_{ij}))$ can be replaced by
a generic inequality $(y \geq f^{\geq}_{ij}))$ where $y$ is a generic variable which stands
for any $x_{j}$, $j \in \{j: (i,j) \in \textsf{argMin}(\theta)\}$. In other words $\textsf{argMin}_{log}(\theta)(i)$ can be
expressed as:

\[
\textsf{argMin}_{log}(\theta)(i) = \textsf{OR}_{j\{:(i,j) \in \textsf{argMin}(\theta)\}} (y \geq f^{\geq}_{ij}).
\]

We just have to generalize the result ((3'), Lemma~\ref{lemmInegalitesgeq}) and we have the
following equivalence to be true.
  \begin{equation}
    \mbox{$\textsf{argMin}_{log}(\theta)(i)$ is true for $\theta'_{i} \in \mathcal{I}_{0}$} \Leftrightarrow \mbox{ $\theta'_{i}$
      is defined by (\ref{eqthmthetaprimei})}.
    \end{equation}
Now, we remark that $\textsf{argMin}(\theta'_{i})$ only contains indices $(i,j)$ which correspond to
  inequalities of the form $y \geq f^{\geq}_{ij}$ which have the same feasability domain $\theta'_{i}$ by definition
  of $\textsf{argMin}$. Thus, from the interval point of view this set of inequalities can be considered as a singleton
  $(y \geq f^{\geq}_{ij})$. Thus, we have to compute the maximum interval such that $\textsf{argMin}_{log}(\tau)$ is true which
  is equivalent to compute the maximum interval such that: $\forall i \in \{i: (i,j) \in \textsf{argMin}(\tau)\}
  (y \geq f^{\geq}_{ij}) \mbox{ is true}$. And by the generalization of the result ((2'), Lemma~\ref{lemmInegalitesgeq})
   we have the following equivalence to be true.

    \begin{equation}
    \mbox{$\textsf{argMin}_{log}(\theta)$ is true for $\theta' \in \mathcal{I}_{0}$} \Leftrightarrow \mbox{ $\theta'$ is defined by
      (\ref{eqthmthetaprime})}.
    \end{equation}

The result for $\theta'$ is now proved in the case (a'') or (b''). The result for $\theta'$ in the case (c'') is proved
  using the same arguments. The only difference deals with the set of indices we consider notably when
  conditioned hierarchy (see Proposition~\ref{propoIneqhBvshUHierarchy}) applies (cf. case (c''. $\B$). 

  The proof of the Theorem is now achieved.

\cqfd. \\

\subsection{The substitution procedure at step $k$ itself}
\label{subsec:substitutionProcedure}
We assume that the substitution process is not finished which means that
the set of the remaining variables $\x$ is not empty. All the results are presented
in the $(z,x,h)$-system of variables. \\

We enumerate the following cases hereafter.

\noi
\textsc{Case 0}. $h=0$ there is no positive maximum,
or $\boldsymbol{\overline{U}}$ (see (\ref{HypprmprBounded})) is false then by Proposition~\ref{propUnbounded}
the \textsf{pmrp} problem is unbounded. \\

\noi
\textsc{Case 1}. Not \textsc{Case 0} and $\x^{+} \neq \emptyset$. \\
$\textsf{cost}^{[k]} \in h\U$ thus Theorem~\ref{thmReachch.h} does not apply. \\

\noi
We apply Proposition~\ref{propomaxsetIleqIgeq}. \\

\noi
We apply Theorem~\ref{thmNTZ} \\
If Theorem~\ref{thmNTZ} provides a $2$-tuple of
functions $(f_{z,i^{*}j^{*}}, f_{i^{*}j^{*}})$ (ie. element
of $S_{\leq}^{[k+1]}$ or $S_{\geq}^{[k+1]}$) then 
Compute the new characteristics of the \textsf{pmrp} problem (see \S~\ref{subsecNewcost+Newconepmrp}). \\
Else \textsf{pmrp} fails. \\

\noi
\textsc{Case 2}. Not \textsc{Case 0} and $\x^{+} = \emptyset$. Necessarily $\textsf{cost}^{[k]} \in h\B$ \\

\noi
\textsc{Case 2.1}. Theorem~\ref{thmReachch.h} applies then \textsf{pmrp} stops and a maximum is reached. \\

\noi
\textsc{Case 2.2}. Theorem~\ref{thmReachch.h} does not apply. \\

\noi
We apply Proposition~\ref{propomaxsetIleqIgeq}. \\

\noi
We apply Theorem~\ref{thmNTZ} \\
If Theorem~\ref{thmNTZ} provides a $2$-tuple of
functions $(f_{z,i^{*}j^{*}}, f_{i^{*}j^{*}})$ (ie. element
of $S_{\leq}^{[k+1]}$ or $S_{\geq}^{[k+1]}$) then 
Compute the new characteristics of the \textsf{pmrp} problem (see \S~\ref{subsecNewcost+Newconepmrp}). \\
Else \textsf{pmrp} fails.

\subsection{The new characteristics elements of the \textsf{pmrp} after
  substitution computed in $\mathcal{O}(mn^{2})$}
  \label{subsecNewcost+Newconepmrp}

  The new set of equalities is defined as follows. \\

  \noi
  If conditions of Theorem~\ref{thmReachch.h} are verified then set $k+1$ to the
  value $n$. We have $\mathcal{L}^{[n]}:= \mathcal{L}^{[k]} \cup_{x_{j} \in \x^{0} \cup \x^{-}}\{x_{j}=0\}$
  and we add the equality: $\{z=c_{h}h\}$. Then,
  $\textsf{pmrp}$
    {\bf stops} and {\bf we solve} by the well-known backtrack substitution the
    linear system $\mathcal{L}^{[n]}$ and we
    express the $x_{j}$'s as function of the homogenization variable $h$. \\

    \noi
    Else,
    we define $\mathcal{L}^{[k+1]}$ as:
    \begin{equation}
      \mathcal{L}^{[k+1]}:= \mathcal{L}^{[k]} \cup \{x_{j^{*}}=f_{i^{*}j^{*}}(x,h)\}
      \end{equation}

  Based on the previous cases the new cost
  function is defined as:

  \begin{equation}
  \textsf{cost}^{[k+1]}(x,h):= f_{z,i^{*}j^{*}}(x,h).
  \end{equation}
  
And in all cases the new
cone $\mathcal{C}(\overline{A}^{[k+1]})$ associated with $\textsf{pmrp}$ is deduced from the
following set of inequalities:

  \begin{subequations}
    \begin{equation}
      \bar{l}
      z \leq \textsf{cost}^{[k+1]}(x,h), \\
      -\textsf{cost}^{[k+1]}(x,h) \leq 0 (\mbox{ie. $z \geq 0$)}, \\
      \forall i \in [|1,m|]\setminus\{i^{*}\}: \; \overline{a}^{[k]}_{i,.} \ell \leq 0, \\
      -f_{i^{*}j^{*}}(x,h) \leq 0 (\mbox{ie. $x_{j^{*}} \geq 0$)}.
      \ear
 \end{equation}
    and $\ell$ is the $n+2$-dimensional column vector which has the same components
    as the vector $w^{[k]}$ except its $j^{*}$th component wich is:
    \begin{equation}
\ell_{j^{*}}:= f_{i^{*}j^{*}}(x,h).
      \end{equation}
  \end{subequations}

  Defining the following $(n+2) \times (n+2)$ {\em transition matrix} $T^{k \rightarrow k+1}$ by:

  \begin{equation}
    \forall j: t^{k \rightarrow k+1}_{j,.}:= \left\{ \bar{lc} e^{\intercal}_{j}  \mbox{ if $j \neq j^{*}$} \\
                                                (0, v^{\intercal}_{i^{*}j^{*}}, r_{i^{*}j^{*}}) & \mbox{ if $j=j^{*}$.}
    \ear \right.
  \end{equation}
  Where $e_{j}$ denotes the $j$th $n+2$-dimensional vector of the
  $(n+2) \times (n+2)$-identity matrix $\bI_{n+2}$.

  And because $\overline{a}_{i^{*},.} T^{k \rightarrow k+1} w= \bzero^{\intercal}$ we also need to define
  the following $(n+2) \times (n+2)$-matrix $\tau^{[k+1]}$ by:

   \begin{equation}
    \forall j: \tau^{[k+1]}_{j,.}:= \left\{ \bar{lc} \bzero^{\intercal}  \mbox{ if $j \neq j^{*}$} \\
                                                -(0, v^{\intercal}_{i^{*}j^{*}}, r_{i^{*}j^{*}}) & \mbox{ if $j=j^{*}$.}
    \ear \right.
  \end{equation}
   The inequality $\tau^{[k+1]} w \leq \bzero$ is equivalent to the nonnegativity condition
   of the substitued variable $x_{j^{*}}$.

   Finally, the new cone $\mathcal{C}(\overline{A}^{[k+1]})$ associated with the \textsf{pmrp} after the
   substitution of variable $x_{j^{*}}$ at step $(k+1)$ has the following characteristics
   computed in  $\mathcal{O}(mn^{2})$ (matrix product):

   \begin{equation}
     (\overline{A},w)^{[k+1]}:= \overline{A}^{[k]} T^{k \rightarrow k+1} + \tau^{[k+1]}, \;
      \left(\bar{c}
     z \\
     \vdots \\
     0 \\
     \vdots \\
     h
     \ear \right) \; \bar{c}
     \mbox{} \\
     \mbox{} \\
     \leftarrow \; j^{*} \\
     \mbox{} \\
     \mbox{}
     \ear 
   \end{equation}

   \section{Fourier-Motzkin elimination method vs substitution method}
   \label{secFourierMotzkinvsSubstit}
   In this section we discuss the differences between the Fourier-Motzkin elimination method
   and the very special substitution method developed in this paper.

   First, let us assume that our substitution method has provided
   the 2-tuple of indexes $(i^{*},j^{*})$ as the output of the Theorem~\ref{thmNTZ}
   associated with the $2$-tuple of functions $(f_{z,i^{*}j^{*}}, f^{\leq}_{i^{*}j^{*}})$.
   Assume that this upper bound of $x_{j^{*}}$, ie. the function $f^{\leq}_{i^{*}j^{*}}$,
   has been choosen among the set of possible upper bounds $\{f^{\leq}_{i^{'}_{1}j^{*}}, \ldots,
   f^{\leq}_{i^{'}_{\overline{n}}j^{*}}\}$. Assume also that the
   variable $x_{j^{*}}$ has also $\underline{n}$ lower bounds $\{f^{\geq}_{i_{1} j^{*}},\ldots, f^{\geq}_{i_{\underline{n}} j^{*}}\}$,
   with $\underline{n} > 1$. Then, we have the following set of inequalities using
   our method:

\begin{equation}
\left\{\bar{c} f^{\geq}_{i_{1} j^{*}}(x,h) \\
\vdots \\
f^{\geq}_{i_{\underline{n}}j^{*}}(x,h)
\ear \right.
\bar{c}
\mbox{} \\
\leq \\
\mbox{}
\ear
\bar{c}
\mbox{} \\
x_{j^{*}}\\
\mbox{}
\ear
\bar{c}
\mbox{} \\
\leq \\
\mbox{}
\ear
\bar{c}
\mbox{} \\
f^{\leq}_{i^{*}j^{*}}(x,h) \\
\mbox{}
\ear
, \; \mbox{for one } i^{*} \in \{i^{'}_{1}, \ldots, i^{'}_{\overline{n}}\},
\end{equation}
which generates $\underline{n} \times (1-1)=0$ additinal inequalities w.r.t
the set of lower bounds. In other words no additional inequality
is required to eliminate the substituable variable $x_{j^{*}}$. \\

On the contrary when using exponential Fourier-Motzkin elimination method we have the
following set of inequalities:

\begin{equation}
\left\{\bar{c} f^{\geq}_{i_{1} j^{*}}(x,h) \\
\vdots \\
f^{\geq}_{i_{\underline{n}} j^{*}}(x,h)
\ear \right.
\bar{c}
\mbox{} \\
\leq \\
\mbox{}
\ear
\bar{c}
\mbox{} \\
x_{j^{*}}\\
\mbox{}
\ear
\bar{c}
\mbox{} \\
\leq \\
\mbox{}
\ear
\left.\bar{c} f^{\leq}_{i^{'}_{1}j^{*}}(x,h) \\
\vdots \\
f^{\leq}_{i^{'}_{\overline{n}}j^{*}}(x,h)
\ear \right\}
\end{equation}
which generates $\underline{n} \times (\overline{n}-1)$ additinal inequalities w.r.t
the set of lower bounds.

Let us also stress the fact that the Fourier-Motzkin elimination method is not
a $\lambda$-parametrized research of optimum which does not take into account the sign of the cost
associated with a variable
of the problem. This is done by our $\lambda$-parametrized research method and also by the
non-parametrized research {\em simplex method}. Moreover,
the Fourier-Motzkin elimination method
does not take into account the hierarchy between variables $\preceq_{var}$
(see (\ref{eqDomfirst})), the partition $(h\B, h\U)$ of the set of non null linear functions involved
in the problem (see Definition~\ref{defhboundedfcts}), the \textsf{positivity property}
(\ref{fijpositivetoutletemps}) of the $f^{\leq}_{ij}$ functions which is not verified
for the $f^{\geq}_{ij}$ functions in general (see Proposition~\ref{propsigngij}) except
for strictly positive functions (see Definition~\ref{defstrictposs}). These latter functions
are also involved in the hierarchy of the variables of the problem and
the definition of dominating variables (see Definition~\ref{variableOrdre}).

   \section{Conclusion}
   \label{secConcl}
   
   We have presented a special substitution method based
   on domain criterion for the variables of the problem and
   a special partition of the non null $(x,h)$-linear functions involved
   in the linear programming problem. For the problem we have treated
   the complexity for finding one variable to be substituted has the same
   order of magnitude as the number of inequalities $I_{\leq}(x_{j},i)$ or $I_{\geq}(x_{j},i)$ to
   be genrated. Adding the computation of the upper bounds on the
   variable $z$ (ie. cost function) the
   complexity is $\mathcal{O}(2mn)$. Then, because we have to substitute $n$ variables
   the complexity of the forward substitution is $\mathcal{O}(2mn^{2})$. We must also
   add $n$ times the complexity of \S~\ref{subsecNewcost+Newconepmrp}: $\mathcal{O}(mn^{3})$.

   The complexity
   of this backward substitution is equal to the complexity of the classical
   backward substitution which is $\mathcal{O}(n)$. But the general procedure
   (\ref{algogenproc}) described in \S~\ref{subGeneralProcedure} can call the
   ``dual''-problem ($m$ variables, $n$ inequalities) which complexity is obviously
   $\mathcal{O}(2m^{2}n) + m^{3}n + m)$. Thus, the overall complexity of the general
   procedure (\ref{algogenproc}) is:

   \begin{equation}
     \label{eqcomplexity}
\mathcal{O}(2mn^{2}+ mn^{3} + n) + \mathcal{O}(2m^{2}n + m^{3}n + m).
     \end{equation}

   The main Theorem~\ref{maintheo} of the
   paper is now proved and also its Corollary~\ref{corSmale9th}.

   \section{Numerical examples}
   \label{exemplePedagogiques}
   In this section we illustrate our strongly polynomial
   method on two examples: postive maximum and negative maximum.
   All the results are presented in the $(z,x,h)$-system
   of variables or in the $(z,y,h)$-system of variables if dual problem
   has to be studied.
   
\subsection{Positive maximum}
\label{subexcas222}
Assume that the substitution of the variable $x_{1}$ at step $k=0$ in
some initial \textsf{pmrp} has produced the following state.

\begin{equation}
  \label{cas222-A1w1}
  (\overline{A}, w)^{[1]}=\left(\bar{cccccc}
  1 & 0 & 1 & 0 & 0 & -3500 \\
  0 & 0 & 1 & 0 & 0 & -3500 \\
  0 & 0 & \frac{1}{10} & -1 & -\frac{1}{10} & -\frac{1}{5} \\
  0 & 0 & -10 & -30 & 2 & 13 \\
  0 & 0 & 0 & -6 & -3 & 4
  \ear \right), \; \left(\bar{c} z \\ 0 \\ x_{2} \\ x_{3} \\ x_{4} \\ h \ear \right).
\end{equation}

We have: $\x^{+}=\emptyset$, $\x^{0}= \{x_{3}, x_{4}\}$, $\x^{-}=\{x_{2}\}$,
$\textsf{cost}^{[1]}= -x_{2} + 0 \times x_{3} + 0 \times x_{4} + 3500 h$ in $h\B$. \\

We apply Proposition~\ref{propomaxsetIleqIgeq} with $D=\emptyset$ (no dominating variables).
and Proposition~\ref{propomaxsetIleqIgeq} returns $\I^{\leq}=\{(2,4)\}$. \\

\noi
\textsc{case 2.2}, \S~\ref{subsec:substitutionProcedure} applies. \\
We trivialy apply Theorem~\ref{thmNTZ} with $\I^{\leq}=\{(2,4)\}$. \\
From the following array of inequalities provided
using Theorem~\ref{propVarsubstitposs}:

\[
\bar{lllll}
\left( \bar{c} z \\ x_{4}
\ear \right)_{2} & \leq & \left( \bar{c} -x_{2} + 3500 h \\
5 x_{2} + 15 x_{3} -\frac{13}{2}h \ear \right) & :: & \left(\bar{c}
h\B \\  h\U \ear \right) \\
\ear
\]
We have: \\
$\tau=\textsf{Nearest-to-zero-criterion-for}(\textsf{Cost}(1,\I^{\leq}))=[-3500 \lambda, 3500 \lambda]$ (see p. \pageref{nearestfct}), case (c). \\
$\textsf{argMin}(\tau)=\{(2,4)\}$, \\
$\tau'= [-\frac{53}{2} \lambda, \frac{53}{2} \lambda]$ and $S_{\leq}^{[2]}=\{(2,4)\}$, see Theorem~\ref{thmNTZ} case (b'). \\
In other words: $(i^{*},j^{*})= (2,4)$.\\

Using \S~\ref{subsecNewcost+Newconepmrp} the new elements of the problem are as follows.

\begin{equation}
  \mathcal{L}^{[2]}= \{x_{4} = 5 x_{2} + 15 x_{3} -\frac{13}{2} h \}.
\end{equation}

Then, we substitute $x_{4}$ (no other choice) in the previous system (\ref{cas222-A1w1}) and
we obtain the new system at step $k=2$:

\begin{equation}
  \label{cas222-A2w2}
  (\overline{A}, w)^{[2]}=\left(\bar{cccccc}
  1 & 0 & 1 & 0 & 0 & -3500 \\
  0 & 0 & 1 & 0 & 0 & -3500 \\
  0 & 0 & -\frac{2}{5} & -\frac{5}{2} & 0 & \frac{9}{20} \\
  0 & 0 & -5 & -15 & 0 & \frac{13}{2} \\
  0 & 0 & -15 & -51 & 0 & \frac{47}{2}
  \ear \right), \; \left(\bar{c} z \\ 0 \\ x_{2} \\ x_{3} \\ 0 \\ h \ear \right).
\end{equation}
We have: $\x^{+}=\emptyset$, $\x^{0}= \{x_{3}\}$, $\x^{-}=\{x_{2}\}$, $\textsf{cost}^{[2]}= -x_{2} + 0 \times x_{3}  + 3500 h$. \\
Theorem~\ref{thmReachch.h} is not verified. \\

\noi
We apply Proposition~\ref{propomaxsetIleqIgeq} with $D=\emptyset$ (no dominating variables).

\noi
\textsc{case 2.2}, \S~\ref{subsec:substitutionProcedure} applies and Proposition~\ref{propomaxsetIleqIgeq} returns \\
$\I^{\geq}=\{(1,2),(1,3), (2,2), (2,3), (3,2), (3,3)\}$.  \\

We apply Theorem~\ref{thmNTZ} with \\
$\I^{\geq}=\{(1,2),(1,3), (2,2), (2,3), (3,2), (3,3)\}$. \\
From the following array of possible inequalities provided
using Theorem~\ref{propVarsubstitposs2}:

\[
\bar{lllll}
\left( \bar{c} z \\ x_{2}
\ear \right)_{1} & \bar{c} \leq \\ \geq \ear & \left( \bar{c}  x_{3} + \frac{27991}{8} h \\
- x_{3} + \frac{9}{8} h\ear \right) & :: & \left(\bar{c}
h\U, \mbox{$[-\frac{27999}{8}\lambda, \frac{27999}{8}\lambda]$} \\  h\B \ear \right) \\
\left( \bar{c} z \\ x_{3}
\ear \right)_{1} & \bar{c} \leq \\ \geq \ear & \left( \bar{c} - x_{2} + 3500 h \\
- x_{2} + \frac{9}{50} h\ear \right) & :: & \left(\bar{c}
h\B, \mbox{$[-3501 \lambda, 3501 \lambda]$} \\  h\B \ear \right) \\
\left( \bar{c} z \\ x_{2}
\ear \right)_{2} & \bar{c} \leq \\ \geq \ear & \left( \bar{c} 3x_{3} + \frac{34987}{10} h \\
-3 x_{3} +  \frac{13}{10} h\ear \right) & :: & \left(\bar{c}
h\U, \mbox{$[-\frac{35017}{10}\lambda, \frac{35017}{10}\lambda]$} \\  h\B \ear \right) \\
\left( \bar{c} z \\ x_{3}
\ear \right)_{2} & \bar{c} \leq \\ \geq \ear & \left( \bar{c} -x_{2} + 3500 h \\
-\frac{1}{3} x_{2} +  \frac{13}{30} h\ear \right) & :: & \left(\bar{c}
h\B,\mbox{$[-3501 \lambda, 3501 \lambda]$} \\  h\B \ear \right) \\
\left( \bar{c} z \\  x_{2}
\ear \right)_{3} & \bar{c} \leq \\ \geq \ear & \left( \bar{c} \frac{51}{15} x_{3} + \frac{104953}{30} h \\  - \frac{51}{15}x_{3} +
\frac{47}{30} h\ear \right) & :: & \left(\bar{c}
h\U,\mbox{$[-\frac{105055}{30}\lambda, \frac{105055}{30}\lambda]$} \\ h\B \ear \right) \\
\left( \bar{c} z \\  x_{3}
\ear \right)_{3} & \bar{c} \leq \\ \geq \ear & \left( \bar{c} -x_{2} + 3500 h \\  - \frac{15}{51}x_{2} +
\frac{47}{102} h\ear \right) & :: & \left(\bar{c}
h\B, \mbox{$[-3501 \lambda, 3501 \lambda]$} \\ h\B \ear \right).
\ear
\]
Note that the computation of the feasible intervals in the above array are done with the set $R_{\lambda}^{\U}(x,h)$
because $\exists f_{z,ij} \in h\U$, eg. $f_{z,12}(x,h)= x_{3} + \frac{27991}{8} h$. \\

But for the computation of $\theta$, the case (c, because $\textsf{cost}^{[2]} \in h\B$) of the function \textsf{Nearest-to-zero-criterion-for} applies (see p. \pageref{nearestfct}) with $\mathcal{T}(2, \I^{\geq})=I^{\geq}$ and we
have: \\
$\theta=\textsf{Nearest-to-zero-criterion-for}(\textsf{Cost}(2,\I^{\geq}))=[-3500 \lambda, 3500 \lambda]$. \\
Then, $\textsf{argMin}(\theta)=\{(1,3), (2,3), (3,3)\}$, \\
and the inequalities associated with this latter set of indices are: $(1,3): \; (x_{3} \geq - x_{2} + \frac{9}{50} h)$,
$(2,3): \; (x_{3} \geq -\frac{1}{3} x_{2} +  \frac{13}{30} h)$ and $(3,3): \; (x_{3} \geq - \frac{15}{51}x_{2} +
\frac{47}{102} h)$. \\

case (b'') of Theorem~\ref{thmNTZ} applies and $\theta'=[-\frac{47}{102} \lambda, \frac{47}{102} \lambda]$ as the maximum
interval among the following intervals: (see (\ref{eqthmthetaprime}, with $(1,3): \; \theta'_{3}=[-\frac{9}{50}\lambda,\frac{9}{50}\lambda]$,
$(2,3): \; \theta'_{3}=[-\frac{13}{30}\lambda,\frac{13}{30}\lambda]$, $(3,3): \;
\theta'_{3}=[-\frac{47}{102}\lambda,\frac{47}{102}\lambda]$). \\
And we have $S_{\geq}^{[3]}=\{(3,3)\}$. \\
In other words: $(i^{*},j^{*})=(3,3)$. \\

Using \S~\ref{subsecNewcost+Newconepmrp} the new elements of the problem are as follows.

\begin{equation}
  \mathcal{L}^{[3]}= \{x_{4} = 5 x_{2} + 15 x_{3} -\frac{13}{2} h \} \cup \{x_{3}= -\frac{15}{51} x_{2} + \frac{47}{102} h\}.
\end{equation}

And we substitute $x_{3}$ in
the previous system (\ref{cas222-A2w2}) and we obtain the following new system.

\begin{equation}
  \label{cas222-A3w3}
  (\overline{A}, w)^{[2]}=\left(\bar{cccccc}
  1 & 0 & 1 & 0 & 0 & -3500 \\
  0 & 0 & 1 & 0 & 0 & -3500 \\
  0 & 0 & \frac{57}{170} & 0 & 0 & -\frac{179}{255} \\
  0 & 0 & -\frac{10}{17} & 0 & 0 & -\frac{7}{17} \\
  0 & 0 & \frac{15}{51} & 0 & 0 & -\frac{47}{102}
  \ear \right), \; \left(\bar{c} z \\ 0 \\ x_{2} \\ 0 \\ 0 \\ h \ear \right).
\end{equation}
We have: $\x^{+}=\emptyset$, $\x^{0} = \emptyset$, $\x^{-}=\{x_{2}\}$, $\textsf{cost}^{[3]}= -x_{2} + 3500 h$. \\

\noi
\textsc{Case 2.1}, \S~\ref{subsec:substitutionProcedure} applies. \\
By Theorem~\ref{thmReachch.h} the function $h \mapsto 3500 h$ is the nonnegative
upper bounding function of $\textsf{cost}^{[3]}$ is reached at $x^{-}=(x_{2})=\bzero$.

Using \S~\ref{subsecNewcost+Newconepmrp} the new elements of the problem are as follows.

\begin{equation}
  \label{eqLfinalmaxpos}
  \mathcal{L}^{[4]}= \{x_{4} = 5 x_{2} + 15 x_{3} -\frac{13}{2} h, x_{3}= -\frac{15}{51} x_{2} + \frac{47}{102}
  h\} \cup \{x_{2}=0\}.
\end{equation}
  
In conclusion, we have to solve the following linear system defined by
$\mathcal{L}^{[4]} \mbox{ (\ref{eqLfinalmaxpos})} \cup \{z=-x_{2} + 3500 h\}$, ie.: 

\[
\left\{\bar{lll}
x_{4} & = & 5 x_{2} + 15 x_{3} -\frac{13}{2} h \\
x_{3} & = & -\frac{15}{51} x_{2} + \frac{47}{102} h \\
z & = & -x_{2} + 3500 h \\
x_{2} & = & 0.
\ear \right.
\]
And by backward substitution it comes:
\begin{equation}
\left\{\bar{lll}
x_{4} & = & \frac{7}{17} h \\
x_{3} & = & \frac{47}{102} h \\
z & = &  3500 h \\
x_{2} & = & 0,
\ear \right.
\end{equation}
Of course $x_{1}$ is unknown because we did
not give its expression as a function of $x_{2}, \ldots, x_{4}$ and
$h$. \\

The procedure LPP (\ref{algogenproc}) has a positive maximum $z=+3500$, ie. Case $1$ applies. \\

  \subsection{Negative maximum}
  \label{subEdB}
We consider the following objective function $z=-x_{1} + x_{2} -3 x_{3}$ which has to be maximized
over polyhedron $\mathcal{P}(A,b)$ associated with its cone by homogenization
$\mathcal{C}(A,b)$ where:

\begin{equation}
  A:=\left(\bar{ccc}
  -2 & 3 & 0 \\
  4 & 1 & 0 \\
  -1 & -3 & 7 \\
  -1 & -1 & -2 \\
  1 & -2 & -3
  \ear \right), \;
  b:= \left(\bar{c}
-1 \\ 7 \\ 29 \\ -6 \\ -4
  \ear \right).
\end{equation}

The polyhedral cone $\mathcal{C}(\overline{A})$ associated with
the $\textsf{pmrp}$ is charectirized by the inequality $\overline{A} w \leq \bzero$ where

\begin{equation}
  \label{edbAwk=0}
  (\overline{A}, w)^{[0]} := \left(\bar{ccccc}
  1 & 1 & -1 & 3 & 0 \\
  0 & 1 & -1 & 3 & 0 \\
  0 & -2 & 3 & 0 & 1 \\
  0 & 4 & 1 & 0 & -7 \\
  0 & -1 & -3 & 7 & -29 \\
  0 & -1 & -1 & -2 & 6 \\
  0 & 1 & -2 & -3 & 4
  \ear \right), \; \left(\bar{c} z \\ x_{1} \\ x_{2} \\ x_{3} \\ h \ear \right).
\end{equation}
We have $\x^{+}=\{x_{2}\}$, $\x^{0}=\emptyset$, $\x^{-}=\{x_{1}, x_{3}\}$,
$\textsf{cost}^{[0]}= -x_{1} + x_{2} - 3x_{3} + 0 \times h$. \\

\noi
We apply Proposition~\ref{propomaxsetIleqIgeq} with $D=\{x_{1}\}$  ($x_{1}$ is dominating
see row $1$ and eg. row $2$ of $\overline{A}^{[0]}$) \\
\noi
and Proposition~\ref{propomaxsetIleqIgeq} returns $\I^{\geq}=\{(1,1)\}$ \\

\noi
\textsc{Case 1}, \S~\ref{subsec:substitutionProcedure} applies. \\

We apply Theorem~\ref{thmNTZ} with $\I^{\geq}=\{(1,1)\}$. \\
From the following array of possible inequalities provided using Theorem~\ref{propVarsubstitposs2}:
\[
\bar{lllll}
\left( \bar{c} z \\ x_{1}
\ear \right)_{1} & \bar{c} \leq \\ \geq \ear & \left( \bar{c} -\frac{1}{2} x_{2}-3x_{3}-\frac{1}{2} h \\
\frac{3}{2} x_{2} + \frac{1}{2} h\ear \right) & :: & \left(\bar{c}
h\U \\  h\U \ear \right)
\ear
\]
We have: \\
$\theta=\textsf{Nearest-to-zero-criterion-for}(\textsf{Cost}(0,\I^{\geq}))=[-4 \lambda, 4 \lambda]$ (see p. \pageref{nearestfct}), case (b). \\
$\textsf{argMin}(\theta)=\{(1,1)\}$, \\
$\theta'=[-\frac{1}{2} \lambda, \frac{1}{2} \lambda]$, $\textsf{argMax}(\theta')=\{(1,1)\}$,
$S_{\geq}^{[1]}=\{(1,1)\}$, see Theorem~\ref{thmNTZ} case (a''). \\
In other words: $(i^{*},j^{*})=(1,1)$. \\

Using \S~\ref{subsecNewcost+Newconepmrp} the new elements of the problem are as follows.
  
\begin{equation}
\mathcal{L}^{[1]}=\{ x_{1} =  \frac{3}{2} x_{2} + \frac{1}{2} h\}.
  \end{equation}

The substitution of $x_{1}$ in the system $(\overline{A}, w)^{[0]}$ defined by (\ref{edbAwk=0}) provides
the new system $(\overline{A}, w)^{[1]}$ defined as follows:
\begin{equation}
  \label{edbAwk=1}
  (\overline{A}, w)^{[1]} := \left(\bar{ccccc}
  1 & 0 & \frac{1}{2} & 3 & \frac{1}{2} \\
  0 & 0 & \frac{1}{2} & 3 & \frac{1}{2} \\
  0 & 0 & -\frac{3}{2} & 0 & -\frac{1}{2} \\
  0 & 0 & 7 & 0 & -5 \\
  0 & 0 & -\frac{9}{2} & 7 & -\frac{59}{2} \\
  0 & 0 & -\frac{5}{2} & -2 & \frac{11}{2} \\
  0 & 0 & -\frac{1}{2} & -3 & \frac{9}{2}
  \ear \right), \; \left(\bar{c} z \\ 0 \\ x_{2} \\ x_{3} \\ h \ear \right).
\end{equation}

$\textsf{checknulvar}(\overline{A}^{[1]})$ returns $x_{2}=x_{3}=0$ and $h=0$ (see
the row vector $\overline{a}^{[1]}_{0,.} \geq \bzero^{\intercal}$). \\
$h=0$: the procedure LPP (\ref{algogenproc}) does not have a positive maximum.

  We go to the ``dual'' problem. We keep the previous notations except for the vector $x$ which is changed
  into the vector $y$ to indicate that we treat the ``dual'' problem.

And the polyhedral cone $\mathcal{C}(\overline{A})$ associated with
the dual \textsf{pmrp} is charactirized by the inequality $\overline{A} w \leq \bzero$ where:

\begin{equation}
  \label{edbdualAwk=0}
  (\overline{A},w)^{[0]}:= \left(\bar{ccccccc}
  1 & -1 & 7 & 29 & -6 & -4 & 0 \\
  0  & -1 & 7 & 29 & -6 & -4 & 0 \\
  0 & 2 & -4 & 1 & 1 & -1 & -1 \\
  0 & -3 & -1 & 3 & 1 & 2 & 1 \\
  0 & 0 & 0 & -7 & 2 & 3 & -3
  \ear \right), \; \left(\bar{c} z \\ y_{1} \\ y_{2} \\ y_{3} \\ y_{4} \\ y_{5} \\ h \ear \right).
  \end{equation}
We have: $\y^{+}=\{y_{1}, y_{4},y_{5}\}$, $\y^{0}=\emptyset$ and $\y^{-}=\{y_{2},y_{3}\}$,
$\textsf{cost}^{[0]}= y_{1}-7y_{2}-29 y_{3} + 6 y_{4} + 4 y_{5} + 0\times h$. \\

We apply Proposition~\ref{propomaxsetIleqIgeq} with $D=\emptyset$ (no dominating variables).

\noi
and Proposition~\ref{propomaxsetIleqIgeq} returns \\
$\I^{\leq}=\{(1,1),(1,4), (2,4), (2,5), (3,4), (3,5)\}$. \\

\noi
\textsc{Case 1}, \S~\ref{subsec:substitutionProcedure} applies. \\
We apply Theorem~\ref{thmNTZ} with $\I^{\leq}=\{(1,1),(1,4), (2,4), (2,5), (3,4), (3,5)\}$. \\
From the following array of possible inequalities provided using Theorem~\ref{propVarsubstitposs}:

\[
\bar{lllll}
\left( \bar{c} z \\ y_{1}
\ear \right)_{1} & \leq & \left( \bar{c} -5y_{2} -\frac{59}{2} y_{3} + \frac{11}{2} y_{4} + \frac{9}{2} y_{5} + \frac{1}{2} h \\
2 y_{2}-\frac{1}{2} y_{3} -\frac{1}{2} y_{4} + \frac{1}{2} y_{5} + \frac{1}{2} h  \ear \right) & :: & \left(\bar{c}
h\U, \mbox{$[-45\lambda,45\lambda]$} \\  h\U \ear \right) \\
\left( \bar{c} z \\ y_{4}
\ear \right)_{1} & \leq & \left( \bar{c} -11 y_{1} + 17y_{2} - 35 y_{3} + 10 y_{5} + 6 h \\
-2 y_{1}+ 4 y_{2} -  y_{3} + y_{5} +  h  \ear \right) & :: & \left(\bar{c}
h\U, \mbox{$[-79\lambda,79\lambda]$} \\  h\U \ear \right) \\
\left( \bar{c} z \\ y_{4}
\ear \right)_{2} & \leq & \left( \bar{c} 19 y_{1} - y_{2} - 47 y_{3} - 8y_{5} - 6 h \\
3 y_{1}+  y_{2} -  3 y_{3} -2 y_{5} -  h  \ear \right) & :: & \left(\bar{c}
h\U, \mbox{$[-81\lambda,81\lambda]$} \\  h\U \ear \right) \\
\left( \bar{c} z \\ y_{5}
\ear \right)_{2} & \leq & \left( \bar{c} 7 y_{1} - 5 y_{2} - 35 y_{3} + 4 y_{4} - 2 h \\
\frac{3}{2} y_{1}+ \frac{1}{2} y_{2} -  \frac{3}{2} y_{3} -\frac{1}{2} y_{4} - \frac{1}{2} h  \ear \right) & :: & \left(\bar{c}
h\U, \mbox{$[-53\lambda,53\lambda]$} \\  h\U \ear \right) \\
\left( \bar{c} z \\ y_{4}
\ear \right)_{3} & \leq & \left( \bar{c}  y_{1} - 7y_{2} - 8 y_{3} - 5 y_{5} + 9 h \\
\frac{7}{2} y_{3} - \frac{3}{2} y_{5} +  \frac{3}{2} h  \ear \right) & :: & \left(\bar{c}
h\U, \mbox{$[-30\lambda,30\lambda]$} \\  h\U \ear \right) \\
\left( \bar{c} z \\ y_{5}
\ear \right)_{3} & \leq & \left( \bar{c} y_{1} - 7 y_{2} -\frac{59}{3} y_{3} + \frac{10}{3} y_{4} + 4 h \\
\frac{7}{3} y_{3} - \frac{2}{3} y_{4} + h  \ear \right) & :: & \left(\bar{c}
h\U, \mbox{$[-35\lambda,35\lambda]$} \\  h\U \ear \right).
\ear
\]
We have: \\
$\tau=\textsf{Nearest-to-zero-criterion-for}(\textsf{Cost}(0,\I^{\leq}))=[-35 \lambda, 35 \lambda]$ (see p. \pageref{nearestfct}
where case (b. $\U$) applies noticing that $\textsf{cost}^{[0]} \in h\U$ and involved functions are in $h\U$). \\
We provide only once to the reader the details of the computation of $\tau$ hereafter. \\
We identify case (b). So we compute interval $\rho$ (see (\ref{eqrho})) as follows. First
we have to compute the $\rho_{i}$'s (see (\ref{eqtrhoi})). From the above array we have: \\
$\rho_{1}= \textsf{Max}([-45\lambda,45\lambda], [-79\lambda,79\lambda])=[-79\lambda,79\lambda]$, \\
$\rho_{2}= \textsf{Max}([-81\lambda,81\lambda], [-53\lambda, 53\lambda])=[-81\lambda,81\lambda]$ \\ and 
$\rho_{3}= \textsf{Max}([-30\lambda,30\lambda], [-35\lambda,35\lambda])=[-35\lambda,35\lambda]$. \\
Thus, we have:
\[
\rho=\textsf{Min}(\rho_{1}, \rho_{2}, \rho_{3})=\textsf{Min}([-79\lambda,79\lambda], [-81\lambda,81\lambda],
[-35\lambda,35\lambda])=[-35\lambda,35\lambda].
\]
We have $\textsf{argMin}(\rho)=\{(3,5)\}$ and $f_{z,35} \in h\U$. So, $\textsf{argMin}(\rho) \cap h\B=\emptyset$ and
there is no conditioned hierarchy in the sense of Proposition~\ref{propoIneqhBvshUHierarchy}. Thus,
case (b. $\U$) applies and $\tau=\rho=[-35\lambda,35\lambda]$.

$\textsf{argMin}(\tau)=\{(3,5)\}$ is a singleton,\\
then $\tau'= [-4 \lambda, 4 \lambda]$ and $S_{\leq}^{[1]}=\{(3,5)\}$, see Theorem~\ref{thmNTZ} case (b', $\U$). \\
In other words: $(i^{*},j^{*})= (3,5)$.\\

Using \S~\ref{subsecNewcost+Newconepmrp} the new elements of the problem are as follows.

\begin{equation}
\mathcal{L}^{[1]}=\{y_{5}= \frac{7}{3} y_{3} - \frac{2}{3} y_{4} + h \}.
\end{equation}

Then, we substitute $y_{5}$ in the previous system (\ref{edbdualAwk=0}) and we obtain the new system:

\begin{equation}
  \label{edbdualAwk=1}
  (\overline{A},w)^{[1]}:= \left(\bar{ccccccc}
  1 & -1 & 7 & \frac{59}{3} & -\frac{10}{3} & 0 & -4 \\
  0 & -1 & 7 & \frac{59}{3} & -\frac{10}{3} & 0 & -4 \\
  0  & 2 & -4 & -\frac{4}{3} & \frac{5}{3} & 0 & -2 \\
  0 & -3 & -1 & \frac{23}{3} & -\frac{1}{3}& 0  & 3 \\
  0 & 0 & 0 & -\frac{7}{3} & \frac{2}{3} & -1
  \ear \right), \; \left(\bar{c} z \\ y_{1} \\ y_{2} \\ y_{3} \\ y_{4} \\ 0 \\ h \ear \right).
\end{equation}
We have: $\y^{+}=\{y_{1},y_{4}\}$, $\y^{0}=\emptyset$ and $\y^{-}=\{y_{2},y_{3}\}$,
$\textsf{cost}^{[1]}= y_{1}-7y_{2}-  \frac{59}{3} y_{3} + \frac{10}{3} y_{4} + 4 h$. \\

\noi
We apply Proposition~\ref{propomaxsetIleqIgeq} with $D=\emptyset$ (no dominating variables) \\

\noi
and Proposition~\ref{propomaxsetIleqIgeq} returns
$\I^{\leq}=\{(1,1), (1,4), (3,4)\}$. \\

\noi
\textsc{Case 1}, \S~\ref{subsec:substitutionProcedure} applies. \\
We apply Theorem~\ref{thmNTZ} with $\I^{\leq}=\{ (1,1), (1,4), (3,4)\}$. \\
From the following array of possible inequalities provided using Theorem~\ref{propVarsubstitposs}:

\[
\bar{lllll}
\left( \bar{c} z \\ y_{1}
\ear \right)_{1} & \leq & \left( \bar{c} -5y_{2} -\frac{57}{3} y_{3} + \frac{5}{2} y_{4} + 5 h \\
2 y_{2} +\frac{2}{3} y_{3} + \frac{5}{2} y_{4} + 5h \ear \right) & :: & \left(\bar{c}
h\U, \mbox{$[-\frac{63}{2}\lambda, \frac{63}{2}\lambda]$} \\  h\U \ear \right) \\
\left( \bar{c} z \\ y_{4}
\ear \right)_{1} & \leq & \left( \bar{c} -3 y_{1} + y_{2} - 17 y_{3} + 8 h \\
-\frac{6}{5} y_{1} + \frac{12}{5} y_{2} +\frac{4}{5} y_{3} + \frac{6}{5}h \ear \right) & :: & \left(\bar{c}
h\U,  \mbox{$[-29 \lambda, 29 \lambda]$} \\  h\U \ear \right) \\
\left( \bar{c} z \\ y_{4}
\ear \right)_{3} & \leq & \left( \bar{c} y_{1} - 7 y_{2} - 8 y_{3} + 9 h \\
\frac{7}{2} y_{3} + \frac{3}{2}h \ear \right) & :: & \left(\bar{c}
h\U,  \mbox{$[-25 \lambda, 25 \lambda]$} \\  h\U \ear \right) 

\ear
\]
We have: \\
$\tau=\textsf{Nearest-to-zero-criterion-for}(\textsf{Cost}(1,\I^{\leq}))=[-25 \lambda, 25 \lambda]$ (see p. \pageref{nearestfct}, case (b. $\U$) noticing that $\textsf{cost}^{[1]} \in h\U$). \\
$\textsf{argMin}(\tau)=\{(3,4)\}$, \\
$\tau'= [-2 \lambda, 2 \lambda]$ and $S_{\leq}^{[2]}=\{(3,4)\}$, see Theorem~\ref{thmNTZ} case (b'). \\
In other words: $(i^{*},j^{*})= (3,4)$.\\

Using \S~\ref{subsecNewcost+Newconepmrp} the new elements of the problem are as follows.

\begin{equation}
\mathcal{L}^{[2]}=\{y_{5}= \frac{7}{3} y_{3} - \frac{2}{3} y_{4} + h \} \cup \{y_{4}=\frac{7}{2} y_{3}  +  \frac{3}{2} h\} 
\end{equation}

Then, we substitute $y_{4}$ (no other choice) in the previous system (\ref{edbdualAwk=1}) and we obtain the new matrix
$\overline{A}$:

\[
\overline{A}:=\left(\bar{ccccccc}
  1 & -1 & 7 & 8 & 0 & 0 & - 9 \\
  0 & -1 & 7 & 8 & 0 & 0 & - 9 \\
  0 & 2 & -4 & \frac{9}{2} & 0 & 0 & \frac{1}{2} \\
  0  & -3 & -1 & \frac{13}{2} & 0 & 0 & \frac{5}{2} \\
  0 & 0 & 0 & -\frac{7}{3} & 0 & 0 & -\frac{3}{2}
  \ear \right).
\]
The new system is obtained by applying the function $\textsf{setrowtozero}(\overline{A})$:

\begin{equation}
  \label{edbdualAwk=2}
  (\overline{A},w)^{[2]}:= \left(\bar{ccccccc}
1 & -1 & 7 & 8 & 0 & 0 & - 9 \\
  0 & -1 & 7 & 8 & 0 & 0 & - 9 \\
  0 & 2 & -4 & \frac{9}{2} & 0 & 0 & \frac{1}{2} \\
  0  & -3 & -1 & \frac{13}{2} & 0 & 0 & \frac{5}{2} \\
  0 & 0 & 0 & 0 & 0 & 0 & 0
  \ear \right), \; \left(\bar{c} z \\ y_{1} \\ y_{2} \\ y_{3} \\ 0 \\ 0 \\ h \ear \right).
\end{equation}
We have: $\y^{+}=\{y_{1}\}$, $\y^{0}=\emptyset$ and $\y^{-}=\{y_{2},y_{3}\}$,
$\textsf{cost}^{[2]}= y_{1} -7 y_{2}- 8 y_{3}  + 9h$. \\

We apply Proposition~\ref{propomaxsetIleqIgeq} with $D=\emptyset$ (no dominating variables) \\

\noi
and Proposition~\ref{propomaxsetIleqIgeq} returns \\
$\I^{\leq}=\{ (1,1)\}$. \\

\noi
\textsc{Case 2.2}, \S~\ref{subsec:substitutionProcedure} applies. \\

We apply Theorem~\ref{thmNTZ} with $\I^{\leq}=\{(1,1)\}$. \\
From the following array of possible inequalities provided
using Theorem~\ref{propVarsubstitposs2}:

\[
\bar{lllll}
\left( \bar{c} z \\ y_{1}
\ear \right)_{1} & \left(\bar{c} \leq \\ \leq \ear \right) & \left( \bar{c} -5 y_{2}-\frac{41}{4} y_{3}  + \frac{35}{4} h \\
2 y_{2} - \frac{9}{4} y_{3} - \frac{1}{4} h  \ear \right) & :: & \left(\bar{c}
h\B \\  h\U \ear \right)
\ear
\]
The result is obvious and we use directly \S~\ref{subsecNewcost+Newconepmrp} to compute
the new elements of the problem which are as follows.

\begin{equation}
  \mathcal{L}^{[3]} = \{y_{5}= \frac{7}{3} y_{3} - \frac{2}{3} y_{4} + h \} \cup
  \{y_{4}=\frac{7}{2} y_{3}  +  \frac{3}{2} h\} \cup \{y_{1}=2 y_{2} - \frac{9}{4} y_{3} - \frac{1}{4} h \}
\end{equation}

Then, we substitute
$y_{1}$ in the previous system (\ref{edbdualAwk=2}) and we obtain the following new system.

\begin{equation}
  \label{edbdualAwk=3}
  (\overline{A},w)^{[3]}:= \left(\bar{ccccccc}
1 & 0 & 5 & \frac{41}{4} & 0 & 0 & - \frac{35}{4}\\
  0 & 0 & 5 & \frac{41}{4} & 0 & 0 & - \frac{35}{4}\\
  0 & 0 & -2 & \frac{9}{4} & 0 & 0 & \frac{1}{4} \\
  0  & 0 & -7 & \frac{53}{4} & 0 & 0 & \frac{13}{4} \\
  0 & 0 & 0 & 0 & 0 & 0 & 0
  \ear \right), \; \left(\bar{c} z \\ 0 \\ y_{2} \\ y_{3} \\ 0 \\ 0 \\ h \ear \right).
\end{equation}
We have: $\y^{+}=\emptyset$,  $\y^{0}=\emptyset$ and $\y^{-}=\{y_{2},y_{3}\}$,
$\textsf{cost}^{[3]}= -4 y_{2} - \frac{41}{4}y_{3} + \frac{35}{4} h$. \\

We apply Proposition~\ref{propomaxsetIleqIgeq} with $D=\{y_{2}\}$.

\noi
and Proposition~\ref{propomaxsetIleqIgeq} returns \\
$\I^{\geq}=\{(1,2),(2,2)\}$. \\

\noi
\textsc{Case 2.2}, \S~\ref{subsec:substitutionProcedure} applies. \\
We apply Theorem~\ref{thmNTZ} with $\I^{\geq}=\{(1,2),(2,2)\}$. \\
From the following array of possible inequalities provided using Theorem~\ref{propVarsubstitposs}:

\[
\bar{lllll}
\left( \bar{c} z \\ y_{2}
\ear \right)_{1} & \left(\bar{c} \leq \\ \geq \ear \right)   & \left( \bar{c} y_{2} - \frac{127}{4} y_{3}  + \frac{65}{4} h \\
\frac{9}{8} y_{3} + \frac{1}{8} h  \ear \right) & :: & \left(\bar{c}
h\B, \mbox{$[-\frac{65}{4} \lambda, \frac{65}{4} \lambda]$} \\ h\U \ear \right) \\
\left( \bar{c} z \\ y_{2}
\ear \right)_{2} & \left(\bar{c} \leq \\ \geq \ear \right) & \left( \bar{c} -\frac{138}{7} y_{3} + \frac{45}{7} h \\
\frac{53}{28} y_{3} + \frac{13}{28} h  \ear \right) & :: & \left(\bar{c}
h\B, \mbox{$[-\frac{45}{7} \lambda, \frac{45}{7} \lambda]$} \\ h\U \ear \right)
\ear
\]
Noticing that $\frac{65}{4} > \frac{35}{4}$ we obviously have: \\
$\theta=\textsf{Nearest-to-zero-criterion-for}(\textsf{Cost}(4,\I^{\geq}))=[-\frac{45}{7} \lambda, \frac{45}{7} \lambda]$
(see p. \pageref{nearestfct}, case (c) applies noticing that $\textsf{cost}^{[3]} \in h\B$). \\
$\textsf{argMin}(\theta)=\{(2,2)\}$, \\
$\theta'=[-\frac{183}{7} \lambda, \frac{183}{7} \lambda]$, $\textsf{argMax}(\theta')=\{(2,2)\}$,
$S_{\geq}^{[4]}=\{(2,2)\}$, see Theorem~\ref{thmNTZ} case (c''). \\
In other words: $(i^{*},j^{*})=(2,2)$. \\

Using \S~\ref{subsecNewcost+Newconepmrp} the new elements of the problem are as follows.

\begin{equation}
  \bar{l}
    \mathcal{L}^{[4]} = \{y_{5}= \frac{7}{3} y_{3} - \frac{2}{3} y_{4} + h \} \cup
    \{y_{4}=\frac{7}{2} y_{3}  +  \frac{3}{2} h\} \cup \\
    \{y_{1}=2 y_{2} - \frac{9}{4} y_{3} - \frac{1}{4} h \} \cup \{
    y_{2}= \frac{53}{28} y_{3} + \frac{13}{28} h\}.
    \ear
\end{equation}

Then, we substitute
$y_{2}$ in the previous system (\ref{edbdualAwk=3}) and we obtain the following new system.

\begin{equation}
  \label{edbdualAwk=4}
  (\overline{A},w)^{[5]}:= \left(\bar{ccccccc}
1 & 0 & 0 & \frac{138}{7} & 0 & 0 & - \frac{45}{7}\\
  0 & 0 & 0 & \frac{138}{7} & 0 & 0 & - \frac{45}{7}\\
  0 & 0 & 0 & -\frac{43}{28} & 0 & 0 & -\frac{19}{28} \\
  0  & 0 & 0 & -\frac{53}{28} & 0 & 0 & -\frac{13}{28} \\
  0 & 0 & 0 & 0 & 0 & 0 & 0
  \ear \right), \; \left(\bar{c} z \\ 0 \\ 0 \\ y_{3} \\ 0 \\ 0 \\ h \ear \right).
\end{equation}
We have: $\y^{+}=\emptyset$,  $\y^{0}=\emptyset$ and $\y^{-}=\{y_{3}\}$,
$\textsf{cost}^{[4]}= - \frac{138}{7} y_{3} + \frac{45}{7} h$. \\

\noi
\textsc{Case 2.1}, \S~\ref{subsec:substitutionProcedure} applies. \\
By Theorem~\ref{thmReachch.h} substitution stops and the function $h \mapsto \frac{45}{7} h$ is the nonnegative
upper bounding function of $\textsf{cost}^{[4]}$ which is reached at
$y^{-}=\bzero$, ie. $y_{3}= 0$. \\

Using \S~\ref{subsecNewcost+Newconepmrp} the new elements of the problem are as follows.

\begin{equation}
  \label{eqLfinalmaxneg}
  \bar{l}
      \mathcal{L}^{[5]} = \{y_{5}= \frac{7}{3} y_{3} - \frac{2}{3} y_{4} + h \} \cup
      \{y_{4}=\frac{7}{2} y_{3}  +  \frac{3}{2} h\} \cup \{y_{1}=2 y_{2} - \frac{9}{4} y_{3} - \frac{1}{4} h \} \cup \\
      \{
    y_{2}= \frac{53}{28} y_{3} + \frac{13}{28} h\} 
      \cup \{y_{3}= 0 \}.
  \ear
\end{equation}
With the homogenization variable $h$ which can take any arbitrary value.

By a simple backward substitution in the triangular system defined by
$\mathcal{L}^{[5]} \mbox{ (\ref{eqLfinalmaxneg})} \cup \{z=-\frac{138}{7} y_{3} + \frac{45}{7} h\}$ one obtains:

\begin{equation}
  \bar{lll}
  y_{3} & = & 0 \\
  y_{2} &=& \frac{13}{28} h \\
  y_{1} &= & 2 \frac{13}{28} h - \frac{1}{4}h =\frac{19}{28} h \\
  y_{4} & = & \frac{3}{2}h \\
  y_{5} & = & -\frac{2}{3} \; \frac{3}{2} h + h = 0 \\
  z & = & \frac{45}{7} h.
  \ear
  \end{equation}
The conclusion is that the procedure LPP (\ref{algogenproc}) has a negative maximum $z=-\frac{45}{7}$, ie. Case $2$ applies. \\

\bibliographystyle{plain}
\bibliography{ref_lt4PL}

\end{document}